\documentclass[reqno,a4paper,11pt]{amsart}
\usepackage[british]{babel}
\usepackage{graphicx}
\usepackage[dvipsnames]{xcolor}
\usepackage{fullpage,hyperref,cleveref,amssymb,mathtools,ifthen,thmtools,tikz,comment,bbm}
\usepackage[shortlabels]{enumitem}
\usepackage{algorithm}
\usepackage{algpseudocode}
\usepackage[all]{hypcap}
\usepackage[group-separator={,}]{siunitx}
\usepackage{booktabs}

\usepackage[backend=biber,giveninits=true,doi=false,url=false,isbn=false,style=trad-plain,hyperref,date=year]{biblatex}
\bibliography{bibliography}

\newtheorem{theorem}{Theorem}[section]
\newtheorem{proposition}[theorem]{Proposition}
\newtheorem{lemma}[theorem]{Lemma}

\newtheorem{claim}[theorem]{Claim}
\newtheorem{conjecture}[theorem]{Conjecture}
\theoremstyle{definition}

\renewcommand{\leq}{\leqslant}
\renewcommand{\le}{\leqslant}
\renewcommand{\geq}{\geqslant}
\renewcommand{\ge}{\geqslant}
\newcommand{\Bin}{\operatorname{Bin}}
\newcommand{\car}{\mathbin{\Box}}
\newcommand{\PC}{\sigma}

\providecommand\given{\nonscript\:\ifthenelse{\equal{\delimsize}{}}{\big\vert}{\delimsize\vert}\nonscript\:\mathopen{}}
\let\Pr\undefined
\DeclarePairedDelimiterXPP\Pr[1]{\mathbb{P}}(){}{#1}
\DeclarePairedDelimiterXPP\Ex[1]{\mathbb{E}}{[}{]}{}{#1}
\DeclarePairedDelimiter{\norm}{\lVert}{\rVert}

\title{Counting independent sets in expanding bipartite regular graphs}
\author{Maur\'icio Collares}
\author{Joshua Erde}
\author{Anna Geisler}
\author{Mihyun Kang}

\address{Institute of Discrete Mathematics, Graz University of Technology, Steyrergasse 30, 8010 Graz, Austria}

\email{\{collares,erde,geisler,kang\}@math.tugraz.at}

\begin{document}

\begin{abstract}
In this paper we provide an asymptotic expansion for the number of independent sets in a general class of regular, bipartite graphs satisfying some vertex-expansion properties, extending results of Jenssen and Perkins on the hypercube and strengthening results of Jenssen, Perkins and Potukuchi. More precisely, we give an expansion of the independence polynomial of such graphs using a polymer model and the cluster expansion. In addition to the number of independent sets, our results yields information on the typical structure of (weighted) independent sets in such graphs. The class of graphs we consider covers well-studied cases like the hypercube or the middle layers graph, and we show further that it includes any Cartesian product of bipartite, regular base graphs of bounded size. To this end, we prove strong bounds on the vertex expansion of bipartite and regular Cartesian product graphs, which might be of independent interest.
\end{abstract}

\maketitle

\section{Introduction}

Independent sets (that is, sets of vertices inducing no edge) are a natural object of study, in part due to the wide range of combinatorial problems which can be phrased in terms of independent sets in graphs, or more generally hypergraphs, see for example \cite{BMS15}.
In particular, there is much interest in describing the number, and typical structure, of independent sets in a graph $G$.
However, counting, or even approximating, the number of independent sets in an arbitrary graph $G$ is known to be a computationally hard problem \cite{PB83,S10}.
Nevertheless, for more restrictive classes of graphs, and in particular for \emph{regular} graphs, much is known about the number, and structure, of independent sets.

Given a graph $G$, we denote by $\mathcal{I}(G)$ the set of independent sets in $G$ and by $i(G) \coloneqq |\mathcal{I}(G)|$ its cardinality, i.e., the number of independent sets in $G$.
Answering a question posed by Granville at the 1988 Number Theory Conference in Banff, Alon~\cite{Al91} proved that $i(G) \leq 2^{(1+o(1))n/2}$ for every $d$-regular $n$-vertex graph $G$.
This has been strengthened by Kahn \cite{Ka01}, for bipartite graphs, and by Zhao \cite{Zh09}, for arbitrary graphs, to
\[
  i(G) \leq (2^{d+1}-1)^{\frac{n}{2d}}.
\]
A well-studied example of such a regular graph is the $t$-dimensional hypercube $Q^t$, that is the graph on the vertex set $\{0, 1\}^t$ with an edge connecting two vertices if and only if they differ in exactly one coordinate.
The independent set problem on $Q^t$ is of particular interest, partly due to its relation to problems from statistical physics and computer science \cite{DGGJ04}.
Asymptotics for $i(Q^t)$ were given by Korshunov and Sapozhenko \cite{KoSa83,Sa1987} (see \cite{Ga19} for a short exposition), showing that
\begin{equation}\label{resultSaKo}
  i(Q^t)=(1+o(1))2\sqrt{e}2^{2^{t-1}},
\end{equation}
and this was improved by Jenssen and Perkins \cite{JePe2020}, who showed that
\begin{equation}\label{resultJePe}
  i(G)=2\sqrt{e} 2^{2^{t-1}} \left(1+ \frac{3t^2-3t-2}{8 \cdot 2^t} + \frac{243t^4 - 646 t^3 - 33 t^2 + 436 t + 76}{384 \cdot 2^{2t}} +O\left(t^6 2^{-3t}\right)\right),
\end{equation}
and gave a description of the general terms in such an expansion, together with an explicit algorithm to compute them. Recently, Kronenberg and Spinka \cite{KrSp22} used similar methods to count the number of independent sets in the \emph{percolated hypercube} $(Q^t)_p$, which is a random sparsification of $Q^t$ where each edge in $Q^t$ is retained with probability $p\in [0,1]$ independently.

For general $d$-regular bipartite graphs satisfying certain expansion assumptions, an approximation of $i(G)$ up to a $(1+o(1)$ multiplicative factor is implicit in the work of Jenssen, Perkins and Potukuchi \cite{JePePo23}, whose work moreover gives a fully polynomial-time approximation scheme (FPTAS) for $i(G)$. Recently, Arras, Garbe and Joos \cite{ArGaJo24} extended these results to regular, $r$-partite, $r$-graphs for $r \geq 3$.

The key ideas underpinning these recent results come from the \emph{hard-core model} which initially arose in the context of statistical physics.
Given a graph $G$, the \emph{independence polynomial} is defined as
\[
  Z(G, \lambda) \coloneqq \sum_{I \in \mathcal{I}(G)} \lambda^{|I|},
\]
where $\lambda>0$ is a parameter known as the \emph{fugacity}.
Note that $Z(G, 1)=i(G)$.

In the \emph{hard-core model} one samples a random independent set $I \in \mathcal{I}(G)$ with probability proportional to $\lambda^{|I|}$.
That is,
\begin{equation}\label{measure-mu}
  \mathbb{P}_{\mu}(\mathbf{I}=I) \coloneqq \frac{\lambda^{|I|}}{Z(G, \lambda)},
\end{equation}
where the independence polynomial is the partition function, i.e., the normalising constant of the resulting probability distribution $\mu$. Kahn \cite{Ka01} initiated the idea of studying the hard-core model on the hypercube $Q^t$ and used it to give information about a `typical' configuration. Extending these ideas Galvin \cite{Ga2011} made use of the graph container method developed by Sapozhenko \cite{Sa1987} and obtained a refined version of \eqref{resultSaKo}.
Jenssen and Perkins \cite{JePe2020} then reformulated this approach in the language of polymer models and introduced a tool from statistical physics, the cluster expansion, to give a formal expansion for $Z(G, \lambda)$, which in the case $\lambda=1$ yields \eqref{resultJePe}. The question of counting independent sets, and more generally homomorphisms, has also been considered in discrete torus $\mathbb{Z}^n_m$ when $m$ is even using related techniques by Jenssen and Keevash \cite{JK23}. Jenssen, Perkins and Potukuchi \cite{JePePo23} also use the cluster expansion to give an approximation to $Z(G, \lambda)$ for regular bipartite graphs satisfying certain vertex-expansion properties, where there is some payoff between the expansion required and the fugacity $\lambda$.

In this paper we provide a more precise asymptotic estimate of $Z(G,\lambda)$, and so in particular of $i(G)$, for regular bipartite graphs satisfying certain vertex-expansion properties. Furthermore, we show that these properties hold in a natural class of graphs which come with a type of product structure, generalising the hypercube and discrete torus.

\subsection{Main results}
As in \cite{ArGaJo24,JePe2020,JePePo23}, in order to derive bounds on the number of independent sets we study the independence polynomial $Z(G, \lambda)$, making extensive use of a polymer model and the cluster expansion. As pointed out by Galvin~\cite{Ga2011}, 
the main ingredients for applying the cluster expansion method are regularity and knowledge about strong enough isoperimetric inequalities of the underlying graph.

The \emph{co-degree} of a graph $G$ is the maximum of $|N(u) \cap N(v)|$ taken over all distinct $u,v \in V(G)$. Our main result is an expansion of the independence polynomial $Z(G, \lambda)$ in $\lambda$, for all $\lambda \geq C_0 \frac{\log^2d}{d^{1/2}}$ and all regular, bipartite graphs with sufficient vertex-expansion properties and with bounded co-degree.

\begin{theorem}\label{thm:general}
For $d,n\in \mathbb N$ let $G$ be a $d$-regular, $n$-vertex bipartite graph with bipartition classes $\mathcal{O}$ and $\mathcal{E}$. Suppose $n=\omega(d^{5})$
and $C_0 \frac{\log^2 d}{d^{1/2}} \leq \lambda \leq C_0$ for some sufficiently large positive constant $C_0$. Suppose $G$ has co-degree at most $\Delta_2 \in \mathbb{N}$ (independent of $d$ and $n$), and that  any set $X \subseteq \mathcal{O}$ or $X \subseteq \mathcal{E}$ satisfies the following vertex-expansion properties:
  \begin{enumerate}[(P1)]
      \item $|N(X)| \geq \frac{32 \log d}{\min(\lambda, 1)} |X|$ if $|X| \leq d^3 \log n$, 
      \item\label{i:main:large-expansion} $|N(X)| \geq (1+\Omega(1/d))|X|$ if $|X|\leq \frac{n}{4}$.
  \end{enumerate}
   Then, for every $j \geq 2$ and $\mathcal{D} \in \{\mathcal{O}, \mathcal{E}\}$, there exist functions $L_j^{\mathcal{D}}(n,d,\lambda)$ and $\varepsilon_j^{\mathcal{D}}(n,d,\lambda)$ satisfying that, as $d \to \infty$,
\[
  |L^{\mathcal{D}}_j|=O\left(\frac{n d^{2(j-1)} \lambda^j}{(1+\lambda)^{jd}}\right)\qquad \textit{and} \qquad
    |\varepsilon^{\mathcal{D}}_j| =O\left(\frac{n d^{2j} \lambda^{j+1} }{(1+\lambda)^{(j+1)d}}\right),
\]
such that, for every $k \in \mathbb{N}$, 
  \[
    Z(G, \lambda) = \left((1+\lambda)e^{\lambda/(1+\lambda)^d}\right)^{n/2}  \sum_{\mathcal{D} \in \{\mathcal{O}, \mathcal{E}\}}\exp\left(\mathbbm{1}_{k\geq 2}\cdot \sum_{j=2}^k L^{\mathcal{D}}_{j} + \varepsilon^{\mathcal{D}}_k\right).
  \]
In particular,  as $d \to \infty$ the number of independent sets in $G$ satisfies
  \[
    i(G)=2^{n/2+1} \exp\left(\frac{n}{2^{d+1}} + O\left(\frac{nd^{2}}{2^{2d}}\right)\right).
  \]
\end{theorem}

Let us make a few remarks about the conditions and conclusions of \Cref{thm:general}. 
It is desirable to approximate the number of independent sets asymptotically, that is up to a multiplicative factor of $(1+o(1))$. If $k \in \mathbb{N}$ is such that $\log_2 n\leq  k d$, then it suffices to consider the first $k$ terms in the expansion. In particular,
    \[
    i(G) = (1+o(1)) \left(2e^{2^{-d}}\right)^{n/2}\sum_{\mathcal{D} \in \{\mathcal{O}, \mathcal{E}\}} \exp\left(\mathbbm{1}_{k\geq 2}\cdot \sum_{j=2}^k L^{\mathcal{D}}_{j}\right),
    \]
    where for $j \leq k$ and $\mathcal{D} \in \{\mathcal{O}, \mathcal{E}\}$, \[|L^{\mathcal{D}}_j|=O\left(\frac{nd^{2(j-1)}}{2^{jd}}\right).\]
Let us discuss briefly, with reference to earlier works, the assumptions and conclusions of \Cref{thm:general}. In comparison to the work of Jenssen, Perkins and Potukuchi \cite{JePePo23}, who obtained a $(1+o(1))$ approximation to $Z(G,\lambda)$ under certain expansion assumptions, our result gives a full expansion of $Z(G,\lambda)$ to arbitrary accuracy, under slightly stronger assumptions on the expansion of the underlying graph. In \cite{JePePo23} there is a payoff between the strength of the expansion and the range of $\lambda$ for which their approximation holds. Namely, if every relevant set expands by at least $(1+\alpha)$ for some constant $\alpha$, then their result holds for $\lambda \geq C_0 \frac{\log d}{d^{1/4}}$, whereas for $\lambda=1$ they require an expansion factor of $(1+\alpha)$ for $\alpha \geq \frac{\log^2 d}{d}$. In fact, by looking carefully at their proofs, it can be seen that their approximation is valid whenever $\alpha \cdot \lambda^2 = \tilde{\Omega}(1/d)$. In comparison, our results hold for the whole range of $\lambda \geq \frac{\log^2 d}{d^{1/2}}$ and $\alpha = \Omega(1/d)$. However, we do require that smaller sets have stronger expansion -- a property that is natural in many commonly studied graph classes, such as product graphs and middle layer graphs. 

This greatly improved range for $\lambda$ is due to a new container lemma by Jenssen, Malekshahian and Park \cite{JeMaPa24}.
In the proof of \Cref{thm:general}, the only place that requires this condition on $\lambda$ is the application of a version of such a container lemma. If an analogous container lemma were to hold for a larger range of $\lambda$, say for $\lambda \geq C_0 \frac{\log d}{d^{c}}$ for some $1/2\le c \le 1$, our proof of \Cref{thm:general} would also hold for such $\lambda$. Indeed, Galvin conjectured that such a container lemma should hold down to $\lambda = \tilde{\Omega}(1/d)$ (see \cite{Ga2011} and the discussion in \Cref{sec:discussion}).

We note that the assumption of \Cref{thm:general} concerning the co-degree
can also be thought of as an expansion assumption. Indeed, it implies that small sets expand almost optimally since, if the co-degree is bounded by $\Delta_2$, then any set $X \subseteq \mathcal{D}$ satisfies
\[
|N(X)| \geq (d-\Delta_2 |X|)|X|.
\]
Furthermore, the co-degree condition also implies that sets which lie inside the neighbourhood of a vertex expand well, which will turn out to play a crucial role in the container lemma.
On the other hand, we ask for much weaker expansion for sets whose size is bounded by
$d^3 \log n$ (\Cref{thm:general}(P1)). The expansion required in this regime depends on the choice of $\lambda$. In particular, if $\lambda$ is a constant, then we require only that sets in this regime expand by $\log d$. However, we note that for any $\lambda$ we consider, an expansion by a factor of $\sqrt{d}$ for sets in this regime is sufficient.

For larger sets, our methods require an expansion of at least $(1+\Omega(1/d))$ (\Cref{thm:general}(P2)). 
This assumption is slightly weaker than that of \cite{JePePo23}, namely by a polylog factor, yet our result is more general, giving an expansion of the partition function for any $\lambda \geq C_0 \frac{\log^2 d}{d^{1/2}}$.
Furthermore, the expansion required for large sets (\Cref{thm:general}(P2)) is weaker than what typical examples, such as the hypercube, exhibit. In particular, the hypercube expands by a factor of $(1+\Omega(1/\sqrt{d}))$.

There is also some restriction on the relationship of the parameters $n$ and $d$ in \Cref{thm:general}. In particular, this can be seen in the first expansion property (\Cref{thm:general}(P1)), which concerns sets of size up to $d^3 \log n$. Typically in examples, the size of the graph is at most exponential in its degree, and we require this property only for sets of polynomial size. However, as $n$ grows, we require this level of expansion for larger and larger sets. We note that, this dependence comes from the particular form of the container lemma, see \Cref{l:container}) and comments thereafter.

The class of graphs covered in \Cref{thm:general} includes several well-known graph classes such as the hypercube or the middle layers graph and implies results on the number of independent sets in these graphs \cite{BaGaLi2021, JePe2020}.
Furthermore, we show that \emph{Cartesian product graphs} of connected, regular, bipartite base graphs of bounded size satisfy the assumptions of \Cref{thm:general}. Given a sequence $(H_i)_{i \in [t]}$ of connected graphs, their \emph{Cartesian product} is the graph $G= \car_{i=1}^t H_i$ with vertex set
 \[
  V(G) \coloneqq \prod_{i=1}^t V(H_i) = \{(x_1,\dots, x_t) \mid x_i \in V(H_i) \text{ for all } i \in [t]\}
\]
and edge set
\[
  E(G) \coloneqq \left\{ \bigl\{ x,y \bigr\} \colon \text{ there is some } i \in [t] \text{ such that } \{x_i,y_i\} \in E(H_i) \text{ and } x_j=y_j \text{ for all }j \neq i \right\}.
\]
The graphs $H_i$ are called the \emph{base graphs} of the Cartesian product graph $G= \car_{i=1}^t H_i$.
When $H_1= \dots =H_t = H$, we call $G= \car_{i=1}^t H_i$ the \emph{$t$-th Cartesian power} of $H$.
Note that if the base graphs $H_i$ are bipartite and $d_i$-regular for all $i \in [t]$, then their Cartesian product graph $G= \car_{i=1}^t H_i$ is also bipartite, and it is $d$-regular with $d=\sum_{i=1}^t d_i$.

Cartesian product graphs include many commonly studied classes of lattice-like graphs, such as hypercubes, grids, tori or Hamming graphs, which are all indeed Cartesian powers of an edge, a path, a cycle or a complete graph, respectively.

In order to apply \Cref{thm:general} to Cartesian product graphs, we establish the following vertex-isoperimetric inequalities in this class of graphs, which may be of independent interest.

\begin{theorem}\label{l:isoperimetry}
Let $m \in \mathbb{N}$ and let $(H_i)_{i \in [t]}$ be a sequence of connected, regular, bipartite graphs with $2 \leq |V(H_i)| \leq m$ for all $i \in [t]$. 
Let $G = \car_{i=1}^t H_i$ with bipartition classes $\mathcal{O},\mathcal{E}$. Then, for every $s \in \mathbb{N}$ there exists $c > 0$ such that  for any $X \subseteq \mathcal{O}$ or $X \subseteq \mathcal{E}$, the following hold in $G$.
  \begin{enumerate}[(a)]
  \item\label{i:isosmall} The co-degree of $G$ is at most $m$.
  \item\label{i:isomedium} If $|X| \leq t^s$, then $|N(X)| \geq ct|X|$.
  \item\label{i:isolarge}  If $|X| = \beta |V(G)|/2$ for some $\beta \in (0,1]$, then
    \begin{equation}\label{e:isolarge}
      |N(X)| \geq \left( 1 + \frac{2 \sqrt{2}(1-\beta)}{m \sqrt{t}}\right) |X|.
    \end{equation}
  \end{enumerate}
\end{theorem}

In particular, bipartite, regular Cartesian product graphs satisfy the assumptions of \Cref{thm:general}. Note that we prove even stronger isoperimetric properties than those required for the proof of \Cref{thm:general} as we show expansion by $\Theta(t)=\Theta(d)$ for sets of polynomial size and show that all sets of size at most $n/4$ expand by a factor of $(1 + \Omega(d^{-1/2}))$ for those graphs.
Thus, \Cref{thm:general} provides an asymptotic expression for the independence polynomial of any regular bipartite Cartesian product graph and taking $\lambda =1$ leads to an estimate for the number of independent sets as a corollary.

In particular, since every Cartesian product graph of the form in \Cref{l:isoperimetry} satisfies $\log n\leq k d$ for some constant $k$, truncating the expansion in \Cref{thm:general} after $k$ terms leads to a $(1+o(1))$-approximation of $i(G)$.

The computation of the terms $L_j^{\mathcal{D}}$ will be discussed and demonstrated by an example in \Cref{sec:productcomplete}. We derive the first two terms in the expansion of $Z(G, \lambda)$ for the Cartesian product graph of the complete bipartite graph $K_{s, s}$, and use them to give a more accurate estimate on $i(G)$ in this case.
While the terms $L_1, L_2$ are computed by hand, we provide an algorithm to compute further terms in \Cref{sec:computation}.

\begin{theorem}\label{thm:bipartiteproduct}
  Let $s \in \mathbb{N}$ and let $G=\car_{i=1}^t K_{s, s}$. Then
  \[
    i(G)=2^{\frac{(2s)^t}{2}+1} \exp\left( \frac{(2 s)^t}{2^{st+1}} +\frac{(2s)^t}{2^{2st+2}} \left( 3\binom{t}{2} + (2^s-1)(s-1)t - 1\right)+ O\left(\frac{(2s)^{t}(st)^{4}}{2^{3st}}\right)\right).
  \]
\end{theorem}
Note that for $G=\car_{i=1}^t K_{s, s}$, we have $\log_2 n \leq d$ and thus the first term of the expansion suffices to get a bound up to $(1+o(1))$ multiplicatively.

\subsection{Key techniques and proof outline}

The independence polynomial $Z(G,\lambda)$ is the partition function for the measure $\mu$ on the set $\mathcal{I}(G)$ of independent sets in $G$, which is defined in \eqref{measure-mu}.
The hard-core model provides a uniform measure on $\mathcal{I}(G)$ at fugacity $\lambda=1$.
For the $t$-dimensional hypercube, Korshunov and Sapozhenko \cite{KoSa83,Sa1987} showed that a `typical' independent set $I \in \mathcal{I}(Q^t)$ is very \emph{unbalanced} across the partition classes, and this perspective was further developed by Kahn \cite{Ka01} and Galvin \cite{Ga2011}.
One of the key ideas in Jenssen and Perkins' work \cite{JePe2020} was that the distribution of the \emph{defect} vertices, i.e., those vertices that are not contained in the majority partition class, can be approximated by a \emph{defect polymer model}, which they analysed using the \emph{cluster expansion}. For a good point of introduction to the cluster expansion see \cite{Jenssen_2024}.
We follow the same strategy as in \cite{JePe2020} in this paper.

That is, instead of working directly with the measure $\mu$ defined in \eqref{measure-mu}, we will introduce a slightly different measure $\hat{\mu}$ on $\mathcal{I}(G)$ (see \Cref{sec:strategy}).
We can think of $\hat{\mu}$ as choosing an independent set $I \in \mathcal{I}(G)$ by first choosing a partition class $\mathcal{D}$ with a particular bias, which we can think of as the choice of a \emph{defect} partition class, which meets our randomly chosen independent set in far fewer vertices than the other partition class.
We then choose the \emph{defect} set $I\cap \mathcal{D}$ via some probability distribution $\nu^{\mathcal{D}}$ according to a polymer model.
Here the choice of polymers and their weights closely follows the strategy in \cite{JePe2020}.
We then choose the rest of the independent set in the majority side by including each vertex which is not forbidden by the defect set independently with probability $q$, for some judicious choice of $q=\lambda(1+\lambda)^{-1}$.
This choice is motivated by the fact, which can be verified by simple calculation, that this is the conditional distribution of $\mu$ if we condition on the choice of a defect set.
By choosing the right scaling, we can associate to $\hat{\mu}$ a partition function $\hat{Z}(G,\lambda)$, which we can express in terms of the partition function $\Xi^{\mathcal{D}}$ of the polymer model.
Our strategy will be to show that we can approximate $\hat{Z}(G,\lambda)$ quite precisely, and that, in turn, $\hat{Z}(G,\lambda)$ is itself a good approximation for $Z(G,\lambda)$.

In order to approximate $\hat{Z}(G,\lambda)$, we analyse the probability distribution $\nu^{\mathcal{D}}$ arising from a polymer model using the cluster expansion, which gives us an expression for $\log \Xi^\mathcal{D}$ as a formal power series over a set of clusters which have an explicit combinatorial interpretation in terms of the graph $G$.
We use the Koteck\'y--Preiss condition (see \cref{sec:KPcond}) to demonstrate the absolute convergence of this power series for a range of $\lambda$, which allows us to approximate $\Xi^\mathcal{D}$ very precisely using the first few terms of this expansion, which we can then lift to an approximation for $\hat{Z}(G,\lambda)$.
In order to verify the Koteck\'y--Preiss condition, we will need strong control over the isoperimetric properties of the graph $G$ and a container-type lemma bounding the weight of large sets with a particular structure. Recently, Jenssen, Malekshahian and Park~\cite{JeMaPa24} proved such a lemma for regular bipartite graphs with weaker assumptions on $\lambda$ compared to previous results~\cite{Ga2011,Sa1987}, leading to quantitative improvements in applications, including ours.

In fact, the same tools that give us convergence can be used to control the typical behaviour of an independent set drawn according to $\hat{\mu}$ and, using this information, we will be able to show that $\hat{\mu}$ is a good approximation to $\mu$, and deduce that $\hat{Z}(G,\lambda)$ is a good approximation to $Z(G,\lambda)$ (see \Cref{sec:propertiesmu}).
Combining this with the approximation for $\hat{Z}(G,\lambda)$ in terms of first few terms of the cluster expansion establishes Theorem \ref{thm:general}.
Applying this at fugacity $\lambda=1$ then leads to an estimate for the number of independent sets.

Using this approximation, computing more terms of the cluster expansion of $\hat{Z}(G,\lambda)$ allows us to calculate more terms in the expansion of $i(G)$ given in \Cref{thm:general}. In order to compute the $k$-th term, one has to characterise all $2$-linked sets on at most $k$ vertices and calculate their neighbourhood sizes. For $G$ being the product of complete, bipartite graphs and $k=2$ this is explicitly demonstrated in \Cref{sec:productcomplete}. 
generalising this gives an algorithm which runs in time $e^{O(k \log k)}$ and polynomial space (see \Cref{sec:computation}).

\subsection{Structure of the paper}

In \Cref{sec:preliminaries} we collect some notation, probabilistic tools and basics about polymer models and the cluster expansion.
In \Cref{sec:strategy} we set up the specific polymer model we use, define a measure $\hat{\mu}$ to sample independent sets and outline the general strategy in more detail.
In \Cref{sec:Kotecky} we verify the Koteck\'y--Preiss condition in order to conclude that the cluster expansion of $\hat{Z}(G,\lambda)$ converges.
Having obtained convergence of the cluster expansion, we are able to show further properties of the measure $\hat{\mu}$ on independent sets in \Cref{sec:propertiesmu}, which are essential to the proof of \Cref{thm:general} which is given in \Cref{sec:proof}.
\Cref{sec:isoperimetry} is dedicated to the proof of \Cref{l:isoperimetry} regarding the vertex expansion of regular and bipartite product graphs.
In order to demonstrate the computation of an expansion from \Cref{thm:general} we compute the first two terms in the expansion of $Z(G, \lambda)$ for the Cartesian power of complete bipartite graphs in \Cref{sec:productcomplete}.
The algorithmic aspects of computing further terms in more generality are discussed in \Cref{sec:computation}.
We conclude in \Cref{sec:discussion} with a discussion of our results and some future directions.

\section{Preliminaries}\label{sec:preliminaries}

Throughout the paper for ease of presentation we will omit floor/ceiling signs and assume that $n$ is even  whenever necessary.
Unless explicitly stated otherwise, all logarithms have the natural base.

\subsection{Probabilistic tools}
Given a non-negative real random variable $X$ and $r \geq 0$, the \emph{cumulant generating function} $K_X(r)$ is defined as
\[
  K_X(r):= \log \mathbb{E}\left[e^{rX}\right].
\]
A simple application of Markov's inequality yields that for any $a>0$ and $r \geq 0$,
\begin{equation}\label{eq:Markov}
  \mathbb{P}(X \geq a) = \mathbb{P}\left(e^{rX} \geq e^{ra}\right) \leq e^{-ra+K_X(r)}.
\end{equation}

Given two probability measures $\mu$ and $\hat{\mu}$ on a sample space $\Omega$, the \emph{total variation distance} between them is defined by
\begin{equation}\label{eq:totalvariationdist}
  \norm{\hat{\mu}-\mu}_{TV} := \frac{1}{2} \sum_{\omega \in \Omega} \left|\hat{\mu}(\omega) - \mu(\omega)\right| = \sum_{\substack{\omega \in \Omega \\ \hat{\mu}(\omega) > \mu(\omega)}} \hat{\mu}(\omega) - \mu(\omega).
\end{equation}

We also use a standard version of the Chernoff bound, see for example \cite[Appendix A]{Alon2016Book}.
\begin{lemma}\label{l:Chernoff}
  Let $N \in \mathbb{N}$, $0 < p < 1$, and $X \sim \Bin(N,p)$. Then for every $r \geq 0$,
  \[
    \Pr[\big]{|X - Np| \geq r} \; \le \; 2\exp\left( -\frac{2r^2}{N} \right).
  \]
\end{lemma}

The following inequality due to Popoviciu (see, e.g., \cite{BhatiaDavis}) will also be useful.
\begin{lemma}\label{l:Popoviciu}
  Let $r, R \in \mathbb{R}$.
  If $X$ is a real random variable such that $r \leq X \leq R$, then
  \[
    \operatorname{Var}(X) \leq \frac{(R-r)^2}{4}.
  \]
\end{lemma}

\subsection{Graph-theoretical tools}
Let $G$ be a graph. Given a set $X \subseteq V(G)$ we denote by $N_G(X) \coloneqq \{v \in V(G) \setminus X \mid \text{ there exists } w \in X \text{ with } vw \in E(G)\}$ the \emph{(external) neighbourhood} of $X$, and we write $N_G(v)$ for $N_G(\{v\})$.
For a vertex $v \in V(G)$, we denote by $d_G(v)$ the \emph{degree} of $v$, that is, $d_G(v):=|N_G(v)|$. For two vertices $v, w \in V(G)$, their co-degree is given by $|N_G(u) \cap N_G(v)|$. The co-degree of $G$ is given by the maximum co-degree of any two distinct vertices in $G$.
When the graph $G$ is clear from context, we will omit the subscripts.
A graph is said to be \emph{$d$-regular} if all vertices have degree $d$. 
We say a set $S\subseteq V(G)$ is \emph{connected (in $G$)} if the induced graph $G[S]$ is connected.
The \emph{second power} of $G$, which we will denote by $G^2$, is the graph with vertex set $V(G)$ in which vertices are adjacent if they are at distance at most $2$ in $G$.
We say a set $S\subset V(G)$ is \emph{$2$-linked} if it is connected in $G^2$, and we will write $N^2(S)$ for the (external) neighbourhood of a set $S$ in $G^2$.

\subsection{Polymer models and the cluster expansion}\label{s:polymer-models}

Let $\mathcal{P}$ be a finite set of objects, called {\em polymers}, which are equipped with a \emph{weight function} $w \colon \mathcal{P} \to \mathbb{C}$ and an \emph{incompatibility relation} $\nsim$ between polymers that is symmetric and reflexive.
We denote by $\Omega$ the collection of pairwise compatible sets of polymers from $\mathcal{P}$.
The {\em polymer model partition function} is defined by
\[
  \Xi^{\mathcal{P}} \coloneqq \sum_{\PC \in \Omega} \prod_{S \in \PC} w(S),
\]
and is the partition function of the probability measure $\nu$ on $\mathcal{P}$ which chooses a set $\PC$ of pairwise compatible polymers with probability proportional to $\prod_{S \in \PC} w(S)$. When the polymer model $\mathcal{P}$ is clear from context, we will just write $\Xi$.

For a tuple $\Gamma = (S_1, \ldots, S_k)$ of polymers, its \emph{incompatibility graph} $H(\Gamma)$ is the graph with vertex set $[k]$, where distinct $i, j \in [k]$ are connected by an edge if $S_i \nsim S_j$.
A \emph{cluster} is an ordered tuple of polymers whose incompatibility graph is connected.
We denote by $\mathcal{C}$ the set of all clusters and for each cluster $\Gamma\in \mathcal{C}$ we set
\begin{equation}\label{eq:defn_cluster_weight}
  w(\Gamma) \coloneqq \phi(\Gamma) \prod_{S \in \Gamma} w(S),
\end{equation}
where
\begin{equation} \label{eq:Ursell}
  \phi(\Gamma) \coloneqq \frac{1}{|\Gamma|!} \sum_{\substack{A \subseteq E(H(\Gamma)) \\ \text{spanning, connected}}} (-1)^{|A|}
\end{equation}
is the \emph{Ursell function}\footnote{We remark that $\phi(\Gamma) \cdot v(H)! = (-1)^{v(H)-1} T_{H(\Gamma)}(1, 0)$, where $T_H$ denotes the Tutte polynomial of $H$.
Since all coefficients of the Tutte polynomial are non-negative, it holds that $\phi(\Gamma) \geq 0$ whenever $\Gamma$ contains an odd number of polymers, and $\phi(\Gamma) \leq 0$ otherwise, with strict inequality in both cases if $H(\Gamma)$ is connected.} of the incompatibility graph of $\Gamma$.

The {\em cluster expansion} is a formal power series
for $\log \Xi^\mathcal{P}$ in the variables $w(S)$, given by
\begin{equation} \label{eq:clusterexpansion}
  \log \Xi^{\mathcal{P}} = \sum_{\Gamma \in \mathcal{C}} w(\Gamma).
\end{equation}
This equality was first observed by Dobrushin \cite{Dob96}, see also \cite{ScottSokal05}.
Note that, whilst the set $\mathcal{P}$ of polymers is finite, the set $\mathcal{C}$ of clusters will not in general be finite, and so we can gain quantitative information about $\Xi^\mathcal{P}$ from \eqref{eq:clusterexpansion} only if the cluster expansion converges.

\subsection{The Koteck\'y--Preiss condition}\label{sec:KPcond}

In this section we discuss a sufficient condition for convergence of the cluster expansion, which provides quantitative tools that can be used to give strong tail bounds.

For a function $g \colon \mathcal{P} \to [0,\infty)$ we also denote by $g$ the corresponding function from $\mathcal{C}$  to $[0, \infty)$ given by $$g(\Gamma) := \sum_{S \in \Gamma} g(S),$$ for any $\Gamma \in \mathcal{C}$.

For a cluster $\Gamma \in \mathcal{C}$ and a polymer $S\in \Gamma$ we write $\Gamma \nsim S$ if there exists $S' \in \Gamma$ so that $S' \nsim S$.

\begin{lemma} [Koteck\'y--Preiss condition]\label{lem:KP}
  Given a set $\mathcal{P}$ of polymers, let $f \colon \mathcal{P} \to [0,\infty)$ and $g \colon \mathcal{P} \to [0,\infty)$ be two functions.  Suppose that for all polymers $S \in \mathcal{P}$,
  \begin{equation}\label{eqKPcond}
    \sum_{S' \nsim S}  |w(S')| e^{f(S') +g(S')}  \le f(S).
  \end{equation}
  Then, the cluster expansion in \eqref{eq:clusterexpansion} converges absolutely.
  In addition, for all polymers $S \in \mathcal{P}$,
  \begin{equation}\label{eqKPtail}
    \sum_{\substack{\Gamma \in \mathcal{C} \\  \Gamma \nsim S}} \left |  w(\Gamma) \right| e^{g(\Gamma)} \le f(S) \,.
  \end{equation}
\end{lemma}

Note, that if the assumptions of the Koteck\'y--Preiss condition hold for some $g \colon \mathcal{P} \to [0,\infty)$, then they also hold for $g=0$.
The benefit of allowing this extra slack in \eqref{eqKPcond} is the improvement in the bounds in \eqref{eqKPtail}, which get stronger as $g$ gets larger. We will use this to control the tails of the cluster expansion.

\subsection{A container lemma for $2$-linked sets}\label{subsec:containers}

Later in the paper, we will consider a polymer model whose polymers are the \emph{$2$-linked} sets in $G$ (see \Cref{sec:strategy}). In this case, in order to obtain quantitative bounds to verify the Koteck\'y--Preiss condition (\Cref{lem:KP}), it will be important to get good bounds on the number of $2$-linked sets.
Note that, for any $d$-regular graph $G$, the maximum degree of the second power $G^2$ is at most $d^2$.
Hence, \cite[Lemma 2]{BFM98} implies the following lemma.

\begin{lemma}\label{l:counting2linked}
  Let $\ell,k \in \mathbb{N}$, let $G$ be a $d$-regular bipartite graph with partition classes $\mathcal{O}, \mathcal{E}$ and let $\mathcal{D} \in \{\mathcal{O}, \mathcal{E}\}$.
  Assume $v \in \mathcal{D}$.
  Then the number of $2$-linked subsets $S \subseteq \mathcal{D}$ of size $\ell$ which contain $v$ is at most $(ed^2)^{\ell-1}$.
\end{lemma}

We will also need to bound certain `weighted' sums over $2$-linked sets, for which we will use the following `container-type' lemma.
Given a bipartite graph $G$ with partition classes $\mathcal{O}, \mathcal{E}$ and $\mathcal{D} \in \{\mathcal{O}, \mathcal{E}\}$,  the closure of a set $A\subseteq \mathcal{D}$ is defined as
\[
  [A] \coloneqq \{v \in \mathcal{D} \mid N(v) \subseteq N(A)\}.
\]
Clearly $A \subseteq [A]$ and $N(A)=N([A])$.
Given $a,b \in \mathbb{N}$
we define
\begin{equation}\label{def:Gab}
  \mathcal{G}^{\mathcal{D}}(a, b) \coloneqq \{A\subseteq \mathcal{D}: A \text{ is 2-linked,}\quad |[A]|=a \quad \text{and}\quad |N(A)|=b\}.
\end{equation}

The following lemma is a consequence of \cite[Lemma 1.2]{JeMaPa24} and
provides an upper bound on the total weight of all sets in $\mathcal{G}^{\mathcal{D}}(a, b)$ for $a$ large enough.
See \Cref{app:container} for a derivation of \Cref{l:container} from the general version in \cite{JeMaPa24}.

\begin{lemma}[a version of Lemma 1.2 in \cite{JeMaPa24}]\label{l:container} For every $\Delta_2 > 0$ and $c > 0$ there exist $C_0>0$ and $d_0>0$ such that the following holds.
  Let $G$ be a connected, $d$-regular, $n$-vertex, bipartite graph with bipartition classes $\mathcal{O}, \mathcal{E}$ with $d \geq d_0$ and co-degree at most $\Delta_2$. If every $X \subseteq \mathcal{O}$ or $X \subseteq \mathcal{E}$ with $|X|\leq n/4$ satisfies \[ |N(X)| \geq \left(1+\frac{c}{d}\right)|X| \]
  and $\lambda$ is at least $C_0 \frac{\log^2 d}{d^{1/2}}$, then for any $d^2 \leq a \leq n/4$ and any $b \geq (1 + c/d)a$ we have
  \begin{equation}\label{e:hyp:container}
    \sum_{A\in \mathcal{G}^{\mathcal{D}}(a, b)}\frac{\lambda^{|A|}}{(1+\lambda)^b}\leq \frac{n}{2}\exp\left(- \frac{(b-a)\log^2 d}{6d}\right).
  \end{equation}
\end{lemma}

Note that \Cref{l:container} holds for any $\lambda \geq C_0 \frac{\log^2 d}{d^{1/2}}$, due to a recent improvement by Jenssen, Malekshahian and Park \cite{JeMaPa24} upon previous results in \cite{Ga2011, Sa1987} holding for a smaller range of $\lambda$. This is also reflected in \Cref{thm:general}.
Furthermore, note that applying this lemma introduces a factor of $n/2$ and this is where the dependence of \Cref{thm:general} on the size of the graph is introduced. In particular, a localised version of this container lemma only taking into account the sets in $\mathcal{G}^{\mathcal{D}}(a, b)$ containing a fixed vertex $v$ and removing the factor of $n$ would make \Cref{thm:general} independent of the size of the graph.

\section{Strategy} \label{sec:strategy}

Throughout the paper we fix a graph $G$ that satisfies the conditions of \Cref{thm:general}, in particular, $G$ is a $d$-regular, $n$-vertex bipartite graph with bipartition classes $\mathcal{O}$ and $\mathcal{E}$, i.e.,
\[
  V(G) = \mathcal{O} \cup \mathcal{E}.
\]
We also note that $G$ is balanced, i.e., $|\mathcal{O}|=|\mathcal{E}| = n/2$. All asymptotics are with respect to $d \to \infty$.

We fix the fugacity $\lambda>0$ and for simplicity we let 
\begin{equation}\label{e:hardcore}
Z=Z(G, \lambda)\coloneqq \sum_{I \in \mathcal{I}(G)} \omega(I),
\end{equation}
with $\omega(I)\coloneqq \lambda^{|I|}$. In addition, we fix a partition class $\mathcal{D} \in \{\mathcal{O},\mathcal{E}\}$ and call it the \emph{defect side}.

We call a set $S \subseteq \mathcal{D}$ such that $S$ is $2$-linked and $|[S]| \leq n/4$ a \emph{$\mathcal{D}$-polymer}.
We denote the set of $\mathcal{D}$-polymers by $\mathcal{P}^{\mathcal{D}}$.
The \emph{weight} of a polymer $S \in \mathcal{P}^{\mathcal{D}}$ is given by
\[
  w(S) \coloneqq \frac{\lambda^{|S|}}{(1+\lambda)^{|N(S)|}}.
\]
We will also extend the weight function $w$ to $2$-linked sets $S$ with $|[S]|>n/4$ by setting $w(S) = 0$. 

We say two polymers $S, S'\in \mathcal{P}^{\mathcal{D}}$ are {\em compatible} if $S \cup S'$ is not $2$-linked, i.e., if their neighbourhoods are disjoint.
We call a set $\PC=\{S_1,\ldots,S_{\ell}\}$ of pairwise compatible $\mathcal{D}$-polymers a {\em polymer configuration} in $\mathcal{D}$ and denote by $\Omega^{\mathcal{D}}$ the set of all polymer configurations in $\mathcal{D}$.

Configurations of $2$-linked sets are closely related to independent sets. Indeed, if we write $I^{\mathcal{D}} \coloneqq I \cap \mathcal{D} = \bigcup_{i=1}^{\ell} S_i \; \subseteq \mathcal{D}$, then a configuration $\PC=\{S_1,\ldots,S_{\ell}\}$ of $2$-linked sets is determined by $I^{\mathcal{D}}$ (by taking the maximal $2$-linked components of $I^{\mathcal{D}}$) and vice versa.
In what follows we will, in an abuse of notation, write $I^{\mathcal{D}} \in \Omega^\mathcal{D}$ to mean that there is a polymer configuration $\PC \in \Omega^{\mathcal{D}}$ such that $I^{\mathcal{D}} = \bigcup_{S \in \PC} S$.

Moreover, for each independent set $I$ in $\mathcal{I}(G)$, the intersection of $I$ with at least one of the bipartition classes gives rise in this way to a polymer configuration.
\begin{proposition}\label{prop:Icaptured}
For each $I \in \mathcal{I}(G)$ we have either $I \cap \mathcal{O} \in \Omega^{\mathcal{O}}$ or $I \cap \mathcal{E} \in \Omega^{\mathcal{E}}$.
\end{proposition}
\begin{proof}
  Assume there exists $I \in \mathcal{I}(G)$ such that the claim does not hold.
  Then, there is a $2$-linked component $S \subseteq I \cap \mathcal{O}$ such that $|[S]| > \frac{n}{4}$ and a $2$-linked component $S' \subseteq I \cap \mathcal{E}$ such that $|[S']| > \frac{n}{4}$.
  
  Since $G$ is $d$-regular and bipartite, using Hall's condition (Theorem 1 in \cite{Hall35}) it contains a perfect matching. Thus, we have $|N(S)|= |N([S])| \geq |[S]| > \frac{n}{4}$. 
  Recalling $|\mathcal{E}|=\frac{n}{2}$ and $N(S) \cup [S'] \subseteq \mathcal{E}$, we obtain $N(S) \cap[S'] \neq \varnothing$ by the pigeonhole principle. So there is a vertex $v$ in $[S']$ that is a neighbour of  a vertex $u$ in $S$. But by definition of the closure, each neighbour of a vertex in $[S']$, such as $u$, is also the neighbour of some vertex in $S'$ (see also \Cref{fig:proofillustration}).  Thus, there is an edge between $S$ and $S'$, contradicting the fact that $I$ is an independent set.   
\end{proof}
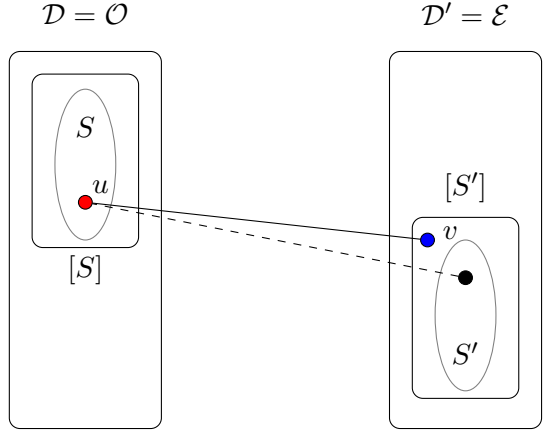
\begin{figure}
  \centering
  \begin{tikzpicture}[scale=1]
    \path[draw=black,rounded corners] (-1,-4) -- (-1,1) -- (1,1) -- (1,-4) -- cycle;
    \path[draw=black,rounded corners] (4,-4) -- (4,1) -- (6,1) -- (6,-4) -- cycle;
    \node[] at (0,1.5) {$\mathcal{D}=\mathcal{O}$};
    \node[] at (5,1.5) {$\mathcal{D}'=\mathcal{E}$};

    \draw[gray, rotate=90] (-0.5,0) ellipse (1cm and 0.4cm);
    \node[] at (0,0) {$S$};
    \path[draw=black,rounded corners] (-0.7,-1.6) -- (0.7,-1.6) -- (0.7,0.7) -- (-0.7,0.7) -- cycle;
    \node[] at (0,-1.9) {$[S]$};
    \draw[gray, rotate=90] (-2.5,-5) ellipse (1cm and 0.4cm);
    \node[] at (5,-3) {$S'$};
    \path[draw=black,rounded corners] (4.3,-3.6) -- (5.7,-3.6) -- (5.7,-1.2) -- (4.3,-1.2) -- cycle;
    \node[] at (5,-0.8) {$[S']$};

    \node (X) at (4.5,-1.5) [circle,draw, fill=blue, scale=0.5] {};
    \node (Y) at (0,-1) [circle,draw, fill=red, scale=0.5] {};
    \node[] at (.2,-0.8) {$u$};
    \node[] at (4.8,-1.4) {$v$};
    \path[draw=black] (X) -- (Y);

    \node (Z) at (5,-2) [circle,draw, fill=black, scale=0.5] {};
    \path[draw=black, dashed] (Z) -- (Y);
  \end{tikzpicture}
  \caption{Illustration of the proof of \Cref{prop:Icaptured}.
  There is a (blue) vertex $v$ in $N(S) \cap [S']$ and a (red) vertex $u$ in $S$ whose neighbour is the vertex $v$. By definition of $[S']$ there exists a (dashed) edge between $S'$ and $S$.}
  \label{fig:proofillustration}
\end{figure}

The $\mathcal{D}$-polymer model (we will write $\Xi^{\mathcal{D}}$ for $\Xi^{\mathcal{P}^{\mathcal{D}}}$) partition function is given by
\begin{equation}\label{e:polymodelpartitionfunction}
  \Xi^{\mathcal{D}}\coloneqq  \sum_{\PC \in \Omega^{\mathcal{D}}} \prod_{S \in \PC} w(S),
\end{equation}
and there is a corresponding measure $\nu^{\mathcal{D}}$ on $\Omega^{\mathcal{D}}$.  For $\sigma \in \Omega^{\mathcal{D}}$ this is given by
\begin{equation} \label{eq:nu}
  \nu^{\mathcal{D}}(\PC) \coloneqq \frac{1}{\Xi^{\mathcal{D}}}{\prod_{S \in \PC} w(S)}.
\end{equation}
Note that, depending on the structure of the graph $G$, the polymer model partition functions $\Xi^{\mathcal{O}}$ and $\Xi^{\mathcal{E}}$ might differ.

As $\nu^{\mathcal{D}}$ is a measure on the set of all configurations of $2$-linked sets, we may think of $\nu^\mathcal{D}$ as a measure on $\{0,1\}^{\mathcal{D}}$, meaning that $\nu^{\mathcal{D}}$ takes value zero if $I^\mathcal{D} \notin \Omega^{\mathcal{D}}$ and the value given above if $I^\mathcal{D} \in \Omega^{\mathcal{D}}$, where we associate each $I^{\mathcal{D}}$ to a configuration of $2$-linked sets as above.
Hence, in what follows we will, in an abuse of notation, write $\nu^\mathcal{D}(I^{\mathcal{D}}) = \nu^\mathcal{D}(\PC)$ if $I^{\mathcal{D}}= \bigcup_{S \in \PC} S$ for some $\PC \in \Omega^{\mathcal{D}}$.
Recall that each subset $I^{\mathcal{D}} \subseteq \mathcal{D}$ is an independent set, since the graph $G$ is bipartite, so $\nu^{\mathcal{D}}$ is a measure on a subset of $\mathcal{I}(G)$.
Given $I^{\mathcal{D}} \in \Omega^{\mathcal{D}}$ associated with a $\mathcal{D}$-polymer configuration $\PC \in \Omega^{\mathcal{D}}$ we get
\begin{equation}\label{e:measurevuD}
  \nu^\mathcal{D}(I^{\mathcal{D}})=\nu^{\mathcal{D}}(\PC) = \frac{1}{\Xi^\mathcal{D}}\prod_{S \in \PC} w(S) = \frac{1}{\Xi^\mathcal{D}}\prod_{S \in \PC} \frac{\lambda^{|S|}}{(1+\lambda)^{|N(S)|}} = \frac{1}{\Xi^\mathcal{D}}\cdot \frac{\lambda^{|I^{\mathcal{D}}|}}{(1+\lambda)^{|N(I^{\mathcal{D}})|}},
  \end{equation}
where the last equality follows since $\sum_{S \in \PC}|S|=|I^{\mathcal{D}}|$ and $\sum_{S \in \PC}|N(S)|=|N(I^{\mathcal{D}})|$.

We define the measure $\hat{\mu}^*$ on $\{\mathcal{O}, \mathcal{E}\} \times \mathcal{I}(G)$ by choosing $(\mathcal{D}, I)$ as follows:
\begin{enumerate}[label=(S\arabic*),ref=S\arabic*]
\item\label{defect-side} pick the \emph{defect} side $\mathcal{D} \in \{\mathcal{O},\mathcal{E}\}$ with probability proportional to $\Xi^{\mathcal{D}}$ and let $\mathcal{D}'$ be the other partition class;
\item\label{defect-step} choose a \emph{defect configuration} $I^{\mathcal{D}}$ according to the probability measure $\nu^\mathcal{D}$ (given in \eqref{e:measurevuD});
\item\label{non-defect-side} choose a set $I^{\mathcal{D}'}$ by including each vertex of $\mathcal{D}' \setminus N(I^\mathcal{D})$ independently with probability $q:=\frac{\lambda}{1+\lambda}$.
\end{enumerate}
This procedure provides two sets $I^{\mathcal{D}} \subseteq \mathcal{D}$ and $I^{\mathcal{D}'} \subseteq \mathcal{D}'$. Setting $I \coloneqq I^{\mathcal{D}} \cup I^{\mathcal{D}'}$ we observe that $I$ is indeed an independent set in $G$.
Furthermore, by \Cref{prop:Icaptured} for every $I \in \mathcal{I}(G)$ at least one of the pairs, $(\mathcal{O}, I)$ or $(\mathcal{E}, I)$, has non-zero measure under $\hat{\mu}^*$.
Hence, for $I \in \mathcal{I}(G)$ we will write 
\[
I^{\mathcal{D}} \coloneqq I \cap \mathcal{D} \quad \text{and}\quad I^{\mathcal{D}'} \coloneqq I \cap \mathcal{D}'.
\]

\begin{figure}
  \centering
  \begin{tikzpicture}[scale=1]
    \path[draw=black,rounded corners] (-1,-4) -- (-1,1) -- (1,1) -- (1,-4) -- cycle;
    \path[draw=black,rounded corners] (4,-4) -- (4,1) -- (6,1) -- (6,-4) -- cycle;
    \node[] at (0,1.5) {$\mathcal{D}=\mathcal{O}$};
    \node[] at (5,1.5) {$\mathcal{D}'=\mathcal{E}$};

    \node (Z1) at (0.5,-1.1) [circle,draw, fill=blue!20, scale=0.5] {};
    \node (Z2) at (0.1,-1.55) [circle,draw, fill=blue!20, scale=0.5] {};
    \node (Z3) at (0.4,-2.2) [circle,draw, fill=blue!20, scale=0.5] {};
    \path[draw=gray,rounded corners] (-0.2,-2.5) -- (-0.2,-0.8) -- (0.9,-0.8) -- (0.9,-2.5) -- cycle;
    \node[] at (0.4,-2.8) {$I^{\mathcal{D}}$};

    \path[draw=gray] (0.7,-0.8) -- (5,-0.5);
    \path[draw=gray] (0.7,-2.5) -- (5,-3.5);
    \draw[gray, rotate=90] (-2,-5) ellipse (1.5cm and 0.4cm);

    \node (Z1) at (4.8,0.5) [circle,draw, fill=red!20, scale=0.5] {};
    \node (Z1) at (5.6,0) [circle,draw, fill=red!20, scale=0.5] {};
    \node (Z1) at (5.3,-3.7) [circle,draw, fill=red!20, scale=0.5] {};
    \node (Z1) at (4.3,-2) [circle,draw, fill=red!20, scale=0.5] {};
    \node (Z1) at (5.5,-1) [circle,draw, fill=red!20, scale=0.5] {};
    \node (Z1) at (5.7,-1.5) [circle,draw, fill=red!20, scale=0.5] {};
    \node (Z1) at (4.7,-0.2) [circle,draw, fill=red!20, scale=0.5] {};
  \end{tikzpicture}
  \caption{The defect configuration $I^{\mathcal{D}}= I \cap \mathcal{D}$ is chosen according to $\nu^{\mathcal{D}}$.
    In $\mathcal{D}'$ vertices in the neighbourhood of $I^{\mathcal{D}}$ are blocked.
    $I \cap \mathcal{D}'$ includes every unblocked vertex, i.e., vertices in $\mathcal{D}' \setminus N(I^{\mathcal{D}})$, independently with probability $\frac{\lambda}{1+\lambda}$.}
\end{figure}
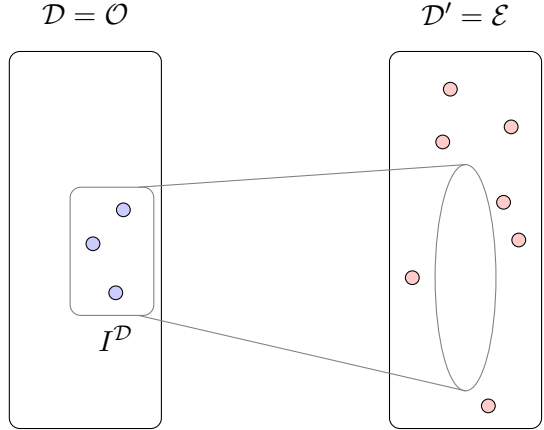

Using \eqref{defect-side}, \eqref{defect-step}, \eqref{non-defect-side}, we can write the measure $\hat{\mu}^*$ explicitly as
\begin{align*}
\hat{\mu}^*(\mathcal{D}, I) &= \frac{\Xi^{\mathcal{D}}}{\Xi^{\mathcal{O}} + \Xi^{\mathcal{E}}} \cdot \nu^{\mathcal{D}}(I^\mathcal{D}) \cdot q^{|I^{\mathcal{D}'}|}\left(1 -q\right)^{n/2 - |N(I^\mathcal{D})| - |I^{\mathcal{D}'}|}.
\end{align*}
Substituting $q = \frac{\lambda}{1+\lambda}$, applying \eqref{e:measurevuD} (using that $\nu^{\mathcal{D}}\left(I^{\mathcal{D}}\right)$ is zero if $I^\mathcal{D} \notin \Omega^{\mathcal{D}}$) and simplifying leads to the following expression
\begin{align}
\hat{\mu}^*(\mathcal{D}, I)&= \frac{\Xi^{\mathcal{D}}}{\Xi^{\mathcal{O}} + \Xi^{\mathcal{E}}} \cdot \frac{\mathbbm{1}_{\{I^\mathcal{D} \in \Omega^{\mathcal{D}}\}}\lambda^{|I^\mathcal{D}|}}{\Xi^\mathcal{D}(1+\lambda)^{|N(I^\mathcal{D})|}} \left( \frac{\lambda}{1+\lambda}\right)^{|I^{\mathcal{D}'}|}\left(\frac{1}{1+\lambda}\right)^{n/2 - |N(I^\mathcal{D})| - |I^{\mathcal{D}'}|} \nonumber\\
&=\left(1+ \lambda \right)^{-n/2}\left(\frac{\mathbbm{1}_{\{I^\mathcal{D} \in \Omega^{\mathcal{D}}\}}}{\Xi^{\mathcal{O}} + \Xi^{\mathcal{E}}}\right) \lambda^{|I|}.\label{eq:muhatstar}
\end{align}
We define the measure $\hat{\mu}$ on $\mathcal{I}(G)$ by
\begin{equation*}
  \hat{\mu}(I) := \hat{\mu}^*(\mathcal{O}, I) + \hat{\mu}^*(\mathcal{E}, I).
\end{equation*}
If we define the weight $\hat{\omega}$ of $I\in \mathcal{I}(G)$ by
\begin{equation}\label{e:hatweight}
  \hat{\omega}(I) := \left(\mathbbm{1}_{\{I^\mathcal{O} \in \Omega^{\mathcal{O}}\}} + \mathbbm{1}_{\{I^\mathcal{E} \in \Omega^{\mathcal{E}}\}}\right) \lambda^{|I|},
\end{equation}
then we can think of the measure $\hat{\mu}$ as arising from a different partition function
\begin{equation}\label{e:hatZdef}
\hat{Z} = \hat{Z}(G,\lambda) \coloneqq \sum_{I \in \mathcal{I}(G)} \hat{\omega}(I) = (1+\lambda)^{n/2}\left( \Xi^{\mathcal{O}} +  \Xi^{\mathcal{E}}\right),
\end{equation}
where we used the fact that for any $I^\mathcal{D} \subseteq \mathcal{D}$
\[
  \sum_{\substack{I \in \mathcal{I}(G) \\ I \cap \mathcal{D} = I^{\mathcal{D}}}} \lambda^{|I|} = \lambda^{|I^{\mathcal{D}}|}\sum_{I^{\mathcal{D}'} \subseteq \mathcal{D}' \setminus N(I^\mathcal{D})} \lambda^{|I^{\mathcal{D}'}|} = (1+\lambda)^{n/2} \frac{\lambda^{|I^{\mathcal{D}}|}}{(1+\lambda)^{|N(I^\mathcal{D})|}}.
\]

Let us compare then the two measures $\mu$ and $\hat{\mu}$ by comparing their partition functions $Z$ and $\hat{Z}$. Given $I \in \mathcal{I}(G)$, \Cref{prop:Icaptured} implies that either $I^\mathcal{O} \in \Omega^{\mathcal{O}}$ or $I^\mathcal{E} \in \Omega^{\mathcal{E}}$ (or both), and we say that $I \in \mathcal{I}(G)$ is \emph{captured by the $\mathcal{D}$-polymer model} if $I^\mathcal{D} \in \Omega^{\mathcal{D}}$.

Hence, recalling that in the hard-core model we have $\omega(I)=\lambda^{|I|}$ (see \eqref{e:hardcore}), it follows by \eqref{e:hatweight} that $\hat{\omega}(I)$ is either equal to $\omega(I)$ or exactly $2 \cdot \omega(I)$, depending on whether $I$ is captured by a single polymer model or both. In other words, the weights $\omega$ and $\hat{\omega}$ agree for all independent sets that are captured by exactly one of the polymer models, whereas $\hat{\omega}$ is twice $\omega$ on all independent sets that are captured by both polymer models.

Our strategy is then twofold: We will approximate $\hat{Z}$ using the \emph{cluster expansion} of $\Xi^\mathcal{D}$ and then bound $\bigl| \log Z - \log \hat{Z}\bigr|$ by controlling the likelihood that an independent set chosen according to $\hat{\mu}$ is captured by both polymer models. In fact, it will turn out that the cluster expansion is useful for this second step as well.

Recall that a cluster is defined as an ordered tuple of polymers whose incompatibility graph is connected. Hence, a cluster $\Gamma$ consists of a tuple of $2$-linked sets whose union is $2$-linked.
We denote by $\mathcal{C}^{\mathcal{D}}$ the set of all clusters.
By \eqref{eq:clusterexpansion} we have the `formal' equality
\[
  \log \Xi^{\mathcal{D}} = \sum_{\Gamma \in \mathcal{C}^{\mathcal{D}}} w(\Gamma).
\]
We hope to understand $\Xi^\mathcal{D}$ via this cluster expansion, and our first step will be to verify the Koteck\'y--Preiss condition for $\Xi^\mathcal{D}$, which we do in \Cref{sec:Kotecky}. Not only will this allow us to verify that the cluster expansion converges, we will also be able to get quantitative bounds on how quickly the tails shrink, and hence be able to quantitatively estimate $\log \Xi^{\mathcal{D}}$ by a truncation of the cluster expansion.

\begin{figure}
  \centering
  \begin{tikzpicture}
    \path[draw=black,rounded corners] (-1,-4) -- (-1,1) -- (1,1) -- (1,-4) -- cycle;
    \path[draw=black,rounded corners] (4,-4) -- (4,1) -- (6,1) -- (6,-4) -- cycle;
    \node[] at (0,1.5) {$\mathcal{D}=\mathcal{O}$};
    \node[] at (5,1.5) {$\mathcal{E}$};

    \draw[blue, rotate=20] (-0.5,-1.2) ellipse (0.8cm and 0.3cm);
    \draw[red, rotate=140] (-1.2, 1.5) ellipse (0.9cm and 0.3cm);
    \draw[color=orange](0,-3) circle (0.2);
    \draw[](0,-3) circle (0.3);
    \node (Z1) at (0.5,-1.1) [circle,draw, fill, scale=0.5] {};
    \node (Z2) at (-0.5,-1.5) [circle,draw, fill, scale=0.5] {};
    \node (Z3) at (0.4,-2.3) [circle,draw, fill, scale=0.5] {};
    \node (Z4) at (0,-3) [circle,draw, fill, scale=0.5] {};
    \path[draw=gray,rounded corners] (-0.9,-3.5) -- (-0.9,-0.8) -- (0.9,-0.8) -- (0.9,-3.5) -- cycle;
    \node[] at (-0.3,-0.5) {$\Gamma$};

    \node (X1) at (5, -1.3) [circle,draw, fill, scale=0.5] {};
    \node (X2) at (5,-1.9) [circle,draw, fill, scale=0.5] {};
    \node (X3) at (5,-2.65) [circle,draw, fill, scale=0.5] {};
    \path[draw=black] (Z1) -- (X1) -- (Z2) -- (X2) -- (Z3) -- (X3) -- (Z4);

    \path[draw=black] (-2.5,1.5) -- (-2.5,-4);
    \node (A1) at (-4,-1) [circle,draw=blue, fill=blue, scale=0.5] {};
    \node (A2) at (-4,-2) [circle,draw=red, fill=red, scale=0.5] {};
    \node (A3) at (-4,-3) [circle,draw=orange, fill=orange, scale=0.5] {};
    \node (A4) at (-5,-3) [circle,draw, fill, scale=0.5] {};
    \path[draw=black] (A1) -- (A2) -- (A3) -- (A4) -- (A2);
    \node (A) at (-4,-3.5) {$H(\Gamma)$};
  \end{tikzpicture}
  \caption{A cluster in $G$ with defect side $\mathcal{O}$.
    The cluster $\Gamma$ depicted here consists of two polymers of size $2$ and twice the same polymer of size one.
    The incompatibility graph $H(\Gamma)$ is depicted on the left side.}\label{fig:cluster}
\end{figure}
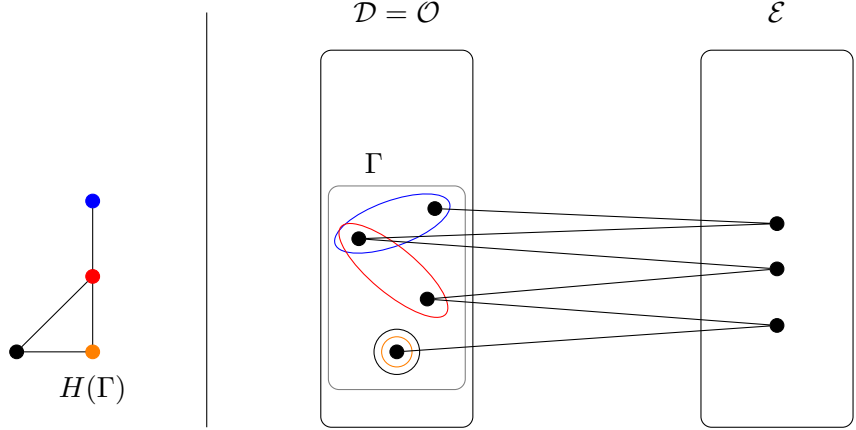

To that end, we define the \emph{size} of a cluster $\Gamma \in \mathcal{C}^{\mathcal{D}}$ to be \[\norm{\Gamma} \coloneqq \sum_{S \in \Gamma} |S|.\]
For $k \in \mathbb{N}$ we let $\mathcal{C}^{\mathcal{D}}_{k}$ be the set of all clusters $\Gamma$ of size $\norm{\Gamma}=k$ and define $\mathcal{C}^{\mathcal{D}}_{< k}$, $\mathcal{C}^{\mathcal{D}}_{\leq k}$, $\mathcal{C}^{\mathcal{D}}_{\geq k}$ and $\mathcal{C}^{\mathcal{D}}_{> k}$ analogously.
Letting
\[
  L^{\mathcal{D}}_{k} \coloneqq \sum_{\Gamma \in \mathcal{C}^{\mathcal{D}}_{k} }w(\Gamma),
\]
denote the $k$-th term of the cluster expansion, we may also define $L^{\mathcal{D}}_{< k}$, $L^{\mathcal{D}}_{\leq k}$, $L^{\mathcal{D}}_{\geq k}$ and $L^{\mathcal{D}}_{> k}$ analogously. 

By~\eqref{eq:clusterexpansion}, truncating the expansion of $\log \Xi^{\mathcal{D}}$ to clusters of size less than $k$ incurs an error of $L^{\mathcal{D}}_{\geq k}$.
Verifying the Koteck\'y--Preiss condition for an explicit choice of $f$ and $g$ will allow us to bound $L^{\mathcal{D}}_{\geq k}$ (see \Cref{l:Kotecky}).
Furthermore, by taking $g$ large enough, we will be able to use the extra information in \eqref{eqKPtail} to control the typical properties of an independent set $I$ drawn according to $\hat{\mu}$ (see \Cref{l:minoritydefect}), and in this way we can bound the difference $\bigl| \log Z - \log \hat{Z}\bigr|$ (see \Cref{l:approximateZ}).
Combining these two estimates will allow us to approximate $Z$ using a truncation of the cluster expansion.

\section{Verifying the Koteck\'y--Preiss condition} \label{sec:Kotecky}

We start by bounding the truncation error $L^{\mathcal{D}}_{\geq k}$. Recall that the co-degree of $G$, the graph whose independent sets are being studied, is at most $\Delta_2$.

\begin{lemma} \label{l:Kotecky}
  Let $k \in \mathbb{N}$ and $C_0 \frac{ \log^2 d}{d^{1/2}} \leq \lambda \leq C_0$ for some sufficiently large constant $C_0>0$. Then
  \[
    |L^{\mathcal{D}}_{\geq k}| \leq \sum_{\Gamma \in \mathcal{C}^{\mathcal{D}}_{\geq k}}|w(\Gamma)| \leq \frac{nd^{8k}}{(1+ \lambda)^{dk - \Delta_2 k^2}}.
  \]
\end{lemma}
For the proof, the following notation will be useful. Let $\lambda_1 \coloneqq \min(\lambda, 1)$ and define a function $\gamma \colon \mathbb{N} \to (1,\infty)$ by
\begin{align}
  \gamma(\ell) &\coloneqq     \left\{\begin{array} {c@{\quad \textup{if} \quad}l}
    (1+ \lambda)^{d - \Delta_2 \ell} \big/ d^8 & \ell \leq d/ \log \log d, \\[+1ex]
    (1+ \lambda)^{16 (\log d)/\lambda_1} & d/ \log \log d < \ell \le d^3 \log n, \\[+1ex]
    \exp(1/d^{2}) & \ell > d^3 \log n.
  \end{array}\right.\label{eq:gamma}
\end{align}

The function $\gamma$ satisfies the following claim, whose simple proof is presented in \Cref{sec:calculations}.
\begin{claim}\label{c:gamma_monotone} The following holds.
\begin{enumerate}[(i)]
  \item \label{c:gamma_nonincreasing} The function $\gamma$ is non-increasing.
  \item \label{c:x_log_gamma_increasing_modulo_discontinuities} For constant $k \in \mathbb{N}$, it holds that $k \log \gamma(k) \leq \ell \log \gamma(\ell)$ for every $k \leq \ell$.
\end{enumerate}
\end{claim}

We will apply the Koteck\'y--Preiss condition with the functions $f,g \colon \mathcal{P}^{\mathcal{D}} \to (0,\infty)$ defined by
\[
  f(S) \coloneqq |S|/d^{2} \qquad \text{ and } \qquad g(S) \coloneqq |S| \log \gamma(|S|).
\]
For every $v \in V(G)$ denote by $\mathcal{P}^{\mathcal{D}}_v$ the set of all polymers $S\in \mathcal{P}^{\mathcal{D}}$ that contain $v$, i.e.,
\[
  \mathcal{P}^{\mathcal{D}}_v:=\{S \in \mathcal{P}^{\mathcal{D}} \mid v \in S\}.
\]

\begin{claim}\label{c:KoteckyPreiss}
  For every $v \in \mathcal{D}$,
  \[
    \sum_{S \in \mathcal{P}^{\mathcal{D}}_v} w(S) e^{f(S) + g(S)} \leq \frac{1}{d^{4}}.
  \]
\end{claim}
Before proving \Cref{c:KoteckyPreiss}, we first show that it implies \Cref{l:Kotecky}.

\begin{proof}[Proof of~\Cref{l:Kotecky}]
  Given $S,S' \in \mathcal{P}^{\mathcal{D}}$, we note that
  \[ S' \nsim S \quad \text{ if and only if }\quad \exists\ v \in N^2[S] \cap S',\]
  where $N^2[S] \coloneqq S \cup N^2(S)$.
  Thus, for a fixed $S \in \mathcal{P}^{\mathcal{D}}$,
  \begin{align} \label{eq:KoteckyFirstBound}
    \sum_{S' \nsim S} w(S') e^{f(S') + g(S')} & \leq \sum_{v \in N^2[S]} \sum_{S' \in \mathcal{P}^{\mathcal{D}}_v} w(S') e^{f(S') + g(S')}.
  \end{align}
  Using $|N^2[S]| \leq d^2 |S|$ and \Cref{c:KoteckyPreiss}, we may upper bound the right-hand side by
  \[
    d^2 |S| \cdot \frac{1}{d^{4}} = \frac{|S|}{d^{2}} = f(S).
  \]
  By the Koteck\'y--Preiss condition (\Cref{lem:KP} \eqref{eqKPtail}) this implies
  \begin{equation}\label{eq:Kotecky}
    \sum_{\substack{\Gamma \in \mathcal{C}^{\mathcal{D}}\\ \Gamma \nsim S}} |w(\Gamma)| e^{g(\Gamma)} \leq f(S),
  \end{equation}
  where $g(\Gamma)\coloneqq \sum_{S \in \Gamma} g(S)$ and $w(\Gamma)$ is defined in \eqref{eq:defn_cluster_weight}.
  For each $v \in V(G)$, if we take $S=\{v\}$, \eqref{eq:Kotecky} becomes
  \[
    \sum_{\substack{\Gamma \in \mathcal{C}^{\mathcal{D}}\\ \Gamma \nsim v}} |w(\Gamma)| e^{g(\Gamma)} \leq \frac{1}{d^{2}},
  \]
  and thus summing over all $v \in V(G)$ we get
  \begin{equation}\label{eq:gamma_cd}
    \sum_{\Gamma \in \mathcal{C}^{\mathcal{D}}} |w(\Gamma)| e^{g(\Gamma)} \leq \frac{n}{d^{2}}.
  \end{equation}
  Recalling that $g(S) \coloneqq |S| \log \gamma(|S|)$, we may bound
  \begin{equation} \label{eq:gandgamma}
    g(\Gamma) = \sum_{S \in \Gamma} g(S)= \sum_{S \in \Gamma} |S| \log \gamma(|S|) \geq \sum_{S \in \Gamma} |S| \log \gamma(\norm{\Gamma}) = \norm{\Gamma} \log \gamma(\norm{\Gamma}),
  \end{equation}
  where the inequality uses \Cref{c:gamma_monotone}\eqref{c:gamma_nonincreasing} and $|S| \leq  \sum_{S \in \Gamma} |S| = \norm{\Gamma}$.
  Applying this observation to~\eqref{eq:gamma_cd} and omitting clusters with $\norm{\Gamma} < k$, we obtain
  \begin{equation} \label{eq:boundrearranged}
      \sum_{\Gamma \in \mathcal{C}^{\mathcal{D}}_{\geq k}} |w(\Gamma)| \cdot e^{\norm{\Gamma} \log \gamma(\norm{\Gamma})} \leq \frac{n}{d^{2}}.
  \end{equation}
    
  For constant $k \in \mathbb{N}$, we may apply \Cref{c:gamma_monotone}\ref{c:x_log_gamma_increasing_modulo_discontinuities} to obtain
  \begin{equation*}
      |L^{\mathcal{D}}_{< k} - \log \Xi| \leq \sum_{\Gamma \in \mathcal{C}^{\mathcal{D}}_{\geq k}} |w(\Gamma)| \leq \frac{n d^{-2}}{\exp(k \log \gamma(k))} \leq \frac{n d^{8k}}{(1+ \lambda)^{dk-\Delta_2 k^2}},
  \end{equation*}
  where the last inequality follows by the definition of $\gamma$ for constant $k$. This proves the lemma.
\end{proof}

We turn to the proof of \Cref{c:KoteckyPreiss}, in which we will assume that $G$ satisfies the isoperimetric inequalities in the statement of \Cref{thm:general}.

\begin{proof}[Proof of \Cref{c:KoteckyPreiss}]
  Fix $v \in \mathcal{D}$. For $\ell \in \mathbb{N}$ let
  \[ \mathcal{Q}_\ell \coloneqq \{S \in \mathcal{P}^{\mathcal{D}}_v \mid |S|=\ell\} \qquad\text{and}\qquad W_\ell := \sum_{S \in \mathcal{Q}_\ell} w(S) e^{f(S)+g(S)}. \]
  Since every $S \in \mathcal{P}^\mathcal{D}_v$ satisfies $|S| \leq |[S]| \leq n/4$, we have that $\mathcal{P}^\mathcal{D}_v = \mathcal{Q}_1 \cup \cdots \cup \mathcal{Q}_{n/4}$. Therefore,
  \[
    \sum_{S \in \mathcal{P}^\mathcal{D}_v} w(S) e^{f(S)+g(S)}= \sum_{\ell=1}^{n/4} W_\ell.
  \]
  Moreover, given $\ell \in \mathbb{N}$, by \Cref{l:counting2linked} we can bound $|\mathcal{Q}_\ell| \leq (ed^2)^{\ell-1} \leq d^{3\ell}$, and therefore
  \begin{equation}\label{eq:c:KoteckyPreiss:max_bound}
  W_\ell \leq d^{3\ell}\cdot \max_{S \in \mathcal{Q}_\ell} w(S) \cdot \max_{S \in \mathcal{Q}_\ell} e^{f(S) + g(S)}.
  \end{equation}
  We will upper bound $W_\ell$ in three different ways according to the value of $\ell$.

  \ \\\noindent\textbf{Case 1: $\ell \leq d/ \log \log d$.}
  For every $S \subseteq \mathcal{D}$, since the co-degree of $G$ is bounded by $\Delta_2$, we have $|N(S)| \geq |S|(d-\Delta_2|S|)$ and thus
  \[\max_{S \in \mathcal{Q}_\ell}  w(S) = \max_{S \in \mathcal{Q}_\ell}  \frac{\lambda^{|S|}}{(1+\lambda)^{|N(S)|}}  \leq \left(\frac{\lambda}{(1+\lambda)^{d - \Delta_2 \ell}}\right)^\ell.\]
  Recalling that $e^{f(S)} = \exp\left(|S| / d^{2}\right)$ and $e^{g(S)} = (\gamma(|S|))^{|S|}$ we get,
  \[
    \max_{S \in \mathcal{Q}_\ell} e^{f(S) + g(S)} = \left(\exp\left(1/d^{2}\right) \cdot \frac{(1+ \lambda)^{d - \Delta_2 \ell}}{d^{8}}\right)^\ell,
  \]
  where we also substituted the definition of $\gamma$ in this range of $\ell$. Plugging these two estimates into~\eqref{eq:c:KoteckyPreiss:max_bound}, we obtain
  \begin{equation}\label{eq:c:KoteckyPreiss:small2}
    \sum_{\ell=1}^{d/ \log \log d} W_\ell \leq \sum_{\ell=1}^{d/ \log \log d} \left(\frac{\lambda \exp\left(1/d^{2}\right)}{d^5}\right)^{\ell}.
  \end{equation}
  Therefore, since $\lambda$ is bounded,  \begin{equation}\label{eq:c:KoteckyPreiss:small_final}
    \sum_{\ell=1}^{d/ \log \log d} W_\ell \leq \sum_{\ell=1}^{\infty} \left(\frac{e\lambda}{d^{5}}\right)^\ell \leq \frac{1}{3d^{4}},
  \end{equation}
  proving the required bound in this case.

  \ \\\noindent\textbf{Case 2: $d/ \log \log d < \ell \leq d^3 \log n$.} Recall that $\lambda_1 \coloneqq\min(\lambda, 1)$. By assumption, we have $|N(S)| \geq 32 \frac{\log d}{\lambda_1}|S|$ for $|S| \leq d^3 \log n$ and thus 
  \[
  \max_{S \in \mathcal{Q}_\ell} w(S)\leq \left(\frac{\lambda}{(1+\lambda)^{\frac{32 \log d}{\lambda_1}}}\right)^\ell.
  \] 
  In this case, $e^{f(S)}= \exp\left(|S|/d^2 \right)$ and $e^{g(S)} = (\gamma(|S|))^{|S|} = (1+\lambda)^{\frac{16|S| \log d}{\lambda_1}}$. Therefore,
  \[
  \max_{S \in \mathcal{Q}_\ell} e^{f(S) + g(S)} = \exp\left(\ell/d^{2}\right) \cdot (1+ \lambda)^{\frac{16 \ell \log d}{\lambda_1}} \leq \left(e \cdot (1+ \lambda)^{\frac{16 \log d}{\lambda_1}}\right)^\ell.
  \]
  Hence, by~\eqref{eq:c:KoteckyPreiss:max_bound} we obtain
  \begin{equation}\label{eq:c:KoteckyPreiss:middle1}
  \sum_{\ell = d/ \log \log d}^{d^3 \log n} W_\ell \leq \sum_{\ell = d/ \log \log d}^{d^3 \log n} \left(\frac{ed^3\lambda}{(1+\lambda)^{\frac{16 \log d}{\lambda_1}}}\right)^{\ell}.
  \end{equation}
  We claim that the above fraction is $o(d^{-4})$. Indeed, if $\lambda \geq 1$, it is $O(d^3 \cdot 2^{-16\log d}) = o(d^{-4})$ for large $d$, since $\lambda$ is also bounded from above. Otherwise, we may bound $1+\lambda \geq e^{\lambda/2}$ and obtain
  \[ 
    \frac{d^3 \lambda}{(1+\lambda)^{\frac{16 \log d}{\lambda}}} = O(d^3 e^{-8 \log d}) = o\left(\frac{1}{d^{4}}\right).
  \]
  Therefore, the geometric sum in~\eqref{eq:c:KoteckyPreiss:middle1} is asymptotically equal to its first term, which is itself $o(d^{-4})$. We conclude that
  \begin{equation}\label{eq:c:KoteckyPreiss:middle_final}
  \sum_{\ell = d/ \log \log d}^{d^3 \log n} W_\ell \leq \frac{1}{3 d^{4}},
  \end{equation}
  as desired.

  \ \\\noindent\textbf{Case 3: $d^3 \log n < \ell \leq n/4$.} 
  For this case, we will use the container lemma instead of~\eqref{eq:c:KoteckyPreiss:max_bound}. 
  
  Let $c > 0$ be such that $|N(X)| \geq (1 + c/d)|X|$ for every $X \subseteq \mathcal{D}$ of size at most $n/4$, as guaranteed by condition \ref{i:main:large-expansion} in  \Cref{thm:general}. 
  Given $a,b \in \mathbb{N}$, let us set
  \[
    \mathcal{Q}_\ell(a, b) \coloneqq \{S \in \mathcal{Q}_\ell \mid |[S]|=a\text{ and }|N(S)|=b\}.
  \]
  Note that for every $a, b \in \mathbb{N}$,
  \[
     \bigcup_{\ell=1}^{a} \mathcal{Q}_{\ell}(a, b) \subseteq \mathcal{G}^{\mathcal{D}}(a,b),
  \]
  where $\mathcal{G}^{\mathcal{D}}(a, b)\coloneqq \{A\subseteq \mathcal{D}: A \text{ 2-linked, } |[A]|=a, |N(A)|=b\}$ as in \eqref{def:Gab}. Setting \[
  \eta \coloneqq 1+ \frac{c}{d},
  \]
  we have $|N(S)| = |N([S])| \geq \eta |[S]|$. Therefore, $\mathcal{Q}_\ell(a,b)$ is empty if $b < \eta a$.
  We may therefore decompose
  \begin{equation}\label{eq:case3weight}
    \sum_{\ell=d^3 \log n}^{n/4} W_\ell \leq \sum_{a=d^3 \log n}^{n/4} \ \sum_{b=\eta a}^{n/2} \ \sum_{\ell=d^3 \log n}^{a} \ \sum_{S \in  \mathcal{Q}_\ell(a, b)} \frac{\lambda^{\ell}}{(1+\lambda)^{b}} \cdot e^{f(S) +g(S)}.
  \end{equation}
  By the container lemma (\Cref{l:container}), for all $d^3 \log n\le a \le \frac{n}{4}$ and every $b \geq \eta a$, we have 
  \begin{equation}\label{eq:c:KoteckyPreiss:container}
  \sum_{\ell=d^3 \log n}^{a} \sum_{ S \in \mathcal{Q}_\ell(a, b)} \frac{\lambda^{\ell} }{(1+\lambda)^{b}} \leq \frac{n}{2} \exp\left(- \frac{(b-a) \log^2 d}{6d}\right).
  \end{equation}
  We may use~\eqref{eq:c:KoteckyPreiss:container} together with the fact that, in this range of $\ell$, $f(S) = g(S) = |S|/d^{2} \leq a/d^{2}$ for every $S \in \mathcal{Q}_\ell(a,b)$, to bound the right-hand side of~\eqref{eq:case3weight} and obtain
  \begin{equation}\label{eq:c:KoteckyPreiss:large2}
    \sum_{\ell=d^3 \log n}^{n/4} W_\ell \leq \sum_{a=d^3 \log n}^{n/4}\  \sum_{b=\eta a}^{n/2} \frac{n}{2} \exp\left(\frac{2a}{d^{2}
    } - \frac{(b-a) \log^2 d}{6d}\right).
  \end{equation}
  Because $a$ and $b$ in \eqref{eq:c:KoteckyPreiss:large2} satisfy $b-a \geq (\eta - 1)a = c a/d$ and $a \geq \ell \geq d^3\log n$, we have
  \[
    \frac{(b-a) \log^2 d}{6d} - \frac{2a}{d^{2}} = \Omega\left(\frac{a \log^2 d}{d^{2}}\right) = \Omega\left(d \log n \cdot \log^2 d\right).
  \]
  Therefore,~\eqref{eq:c:KoteckyPreiss:large2} reduces to
  \begin{equation}\label{eq:c:KoteckyPreiss:large_final}
    \sum_{\ell=d^3 \log n}^{n/4} W_\ell \leq \exp\left(-\Omega\left(d \log n \cdot \log^2 d\right)\right) \leq \frac{1}{3d^{4}}.
  \end{equation}
  
  Combining \eqref{eq:c:KoteckyPreiss:small_final}, \eqref{eq:c:KoteckyPreiss:middle_final} and \eqref{eq:c:KoteckyPreiss:large_final} yields \Cref{c:KoteckyPreiss}.
\end{proof}

\section{Properties of the measure $\hat{\mu}$} \label{sec:propertiesmu}

We start by describing the typical structure of the defect configuration drawn according to the measure $\nu^{\mathcal{D}}$.

\begin{lemma}\label{l:Gammasmall}
  Let $\mathcal{D} \in \{\mathcal{O},\mathcal{E}\}$ and let $ I^\mathcal{D}
  $ be drawn according to $\nu^\mathcal{D}$ as in \eqref{e:measurevuD}.
  Then
  \[
    \mathbb{P}_{\nu^\mathcal{D}}\left( | I^\mathcal{D}| \geq \frac{n}{d^2} \right) \leq\exp\left(-\frac{n}{2d^{4}}\right).
  \]
\end{lemma}
\begin{proof}
  Let us define an auxiliary weight function $\tilde{w} \colon \mathcal{P}^\mathcal{D} \to (0,\infty)$ by
  \[
    \tilde{w}(S) :=  w(S) \exp(|S|/d^{2}).
  \]
  Analogously to the expansion of $\log \Xi^{\mathcal{D}}$, the cluster expansion in $\tilde{w}(S)$ is given by
  \[
    \log \tilde{\Xi}^{\mathcal{D}}= \sum_{\Gamma \in \mathcal{C}^{\mathcal{D}}} \tilde{w}(\Gamma),
  \]
  where $\tilde{\Xi}^{\mathcal{D}}$ is the partition function associated with $\tilde{w}$ as in \eqref{e:polymodelpartitionfunction}. The statements of \eqref{eq:c:KoteckyPreiss:small_final}, \eqref{eq:c:KoteckyPreiss:middle_final} and \eqref{eq:c:KoteckyPreiss:large_final} equally hold for $\tilde{w}(S)$
  and thus the Koteck\'y--Preiss condition holds for $\tilde{w}(S)$, $f(S)$ and $g(S)$.
  For any $\PC \in \Omega^\mathcal{D}$,
  \begin{equation}\label{eq:omegatilde}
    \prod_{S \in \PC} \tilde{w}(S) = \prod_{S \in \PC} w(S) \exp(|S|/d^{2}) = \exp(\norm{\PC} / d^{2}) \prod_{S \in \PC} w(S),
  \end{equation}
  where $\norm{\PC} := \bigcup_{S\in \PC} |S|$.
  Recalling the definition of $\Xi^{\mathcal{D}}$ from \eqref{e:polymodelpartitionfunction}, we thus conclude
  \begin{equation*}
    \frac{\tilde{\Xi}^\mathcal{D}}{\Xi^\mathcal{D}}=\frac{\sum_{\PC \in \Omega^\mathcal{D}} \exp\left(\norm{\PC} / d^{2}\right) \prod_{S \in \PC} w(S)}{\Xi^{\mathcal{D}}}\\
    = \sum_{\PC \in \Omega^\mathcal{D}}  \exp\left(\norm{\PC} / d^{2}\right) \nu^\mathcal{D}(\PC).
    \end{equation*}
    Therefore, $\tilde{\Xi}^\mathcal{D}/ \Xi^\mathcal{D} = \mathbb{E}_{\nu^\mathcal{D}} [ \exp(|I^\mathcal{D}| / d^{2}) ]$. Recalling that the cumulant generating function of a random variable $X$ is $K_X(r):= \log \mathbb{E}[e^{rX}]$, we obtain
  \begin{equation}\label{e:tildeXilower} K_{|I^\mathcal{D}|}\left(1/d^{2}\right) = 
 \log \tilde{\Xi}^\mathcal{D}  -  \log {\Xi}^\mathcal{D} \leq \log \tilde{\Xi}^\mathcal{D},
  \end{equation}
  where the inequality is just the statement that $\Xi^\mathcal{D} \geq 1$ (the empty polymer configuration has weight $1$).
  By the Koteck\'y--Preiss condition the analogue of \eqref{eq:boundrearranged} holds for $\tilde{w}(S)$ and applying this with $k=1$, noting that $\mathcal{C}^{\mathcal{D}}_{\geq 1} = \mathcal{C}^{\mathcal{D}}$, yields
  \[
    \log \tilde{\Xi}^\mathcal{D} \leq  \sum_{\Gamma \in \mathcal{C}^{\mathcal{D}}} |\tilde{w}(\Gamma)| \leq \frac{n d^{-2}}{\gamma(1)} = \frac{n d^{6}}{(1+ \lambda)^{d - \Delta_2}} = O\left(\frac{nd^6}{(1+\lambda)^d}\right),
  \]
  where the second equality follows from the definition of $\gamma(1)$ in \eqref{eq:gamma}, and the last follows from $\lambda = O(1)$. Hence, by \eqref{e:tildeXilower}, we have
  \begin{equation}\label{eq:boundcumulant}
    K_{|I^\mathcal{D}|}\left(1/d^2\right) = O\left(\frac{nd^6}{(1+\lambda)^d}\right),
  \end{equation}
  and we can use Markov's inequality in the form of \eqref{eq:Markov} with $r=1/d^2$ to bound the upper tail of $|I^\mathcal{D}|$ and obtain
  \[
    \mathbb{P}\left(|I^\mathcal{D}|\geq \frac{n}{d^2}\right) 
    \leq \exp\left(-\frac{n}{d^{4}} + K_{|I^\mathcal{D}|}\left(\frac{1}{d^{2}}\right)\right)
  \leq \exp\left(-\frac{n}{d^{4}} + O\left(\frac{nd^6}{(1+\lambda)^d}\right)\right),
  \]
  due to \eqref{eq:boundcumulant}. 
  Since $\lambda \geq C_0 (\log^2 d)/\sqrt{d}=\omega\big((\log d)/d\big)$, the term $(1+\lambda)^d$ grows faster than any fixed polynomial in $d$ and in particular $nd^6/(1+\lambda)^{d}=o(n/d^4)$. For $d$ large enough this yields
  \[
    \mathbb{P}\left(|I^\mathcal{D}|\geq \frac{n}{d^2}\right) \leq \exp\left(-\frac{n}{2d^{4}}\right),
  \]
  as claimed.
\end{proof}

In other words, \Cref{l:Gammasmall} says that typically the defect configuration will contain just a small fraction of the vertices on the defect side.
Thus, there are many unblocked vertices on the non-defect side, i.e., vertices in $\mathcal{D}' \setminus I^{\mathcal{D}}$ that may be included in $I^{\mathcal{D}'}$ in \hyperref[non-defect-side]{(S3)}.

For an independent set $I\in \mathcal{I}(G)$, let $\mathcal{M}=\mathcal{M}(I)$ be its minority side, i.e., the element $\mathcal{M} \in \{ \mathcal{O},\mathcal{E}\}$ such that $|I \cap \mathcal{M}| \leq |I \setminus \mathcal{M}|$ (in case of equality, we choose one side arbitrarily).
The next lemma concludes that typically the defect side $\mathcal{D}$ contains fewer vertices of the independent set than the non-defect side $\mathcal{D}'$.

\begin{lemma}\label{l:minoritydefect}
  Let $(\mathcal{D}, I)$ be drawn according to the measure $\hat{\mu}^*$ as in \eqref{eq:muhatstar}, and let $\mathcal{M}(I)$ be the minority side of $I$.
  Then
  \[
    \mathbb{P}_{\hat{\mu}^*}\big( \mathcal{D} \neq \mathcal{M}(I)\big) \leq \exp\left(-\frac{n}{3d^{4}}\right).
  \]
\end{lemma}
\begin{proof}
  In the following, we condition on the event that the defect set $I^\mathcal{D}$ chosen in \hyperref[defect-step]{(S2)} fulfils $|I^\mathcal{D}| \leq \frac{n}{d^2}$ and note that \Cref{l:Gammasmall} bounds the failure probability for this event by
  \begin{equation} \label{eq:failure}
    \mathbb{P}_{\nu^\mathcal{D}}\left( | I^\mathcal{D}| \geq \frac{n}{d^2} \right) \leq \exp\left(-\frac{n}{2d^{4}}\right).
  \end{equation}

  Let $X = |I \setminus \mathcal{D}|=|I \cap \mathcal{D}'|$ be the size of the intersection of $I$ with the non-defect side.
  By \hyperref[non-defect-side]{(S3)}, $X$ is distributed as $\Bin(M, \frac{\lambda}{1+\lambda})$, where $M \coloneqq \frac{n}{2}- |N(I^\mathcal{D})| \geq \frac{n}{2} - d |I^\mathcal{D}|$ and hence, under our conditioning, $X$ is stochastically dominated by a random variable $Y \sim \Bin(M', \frac{\lambda}{1+\lambda})$ where $M' =\frac{n}{2}\left(1-\frac{2}{d}\right)$, and, in particular,
  \begin{equation}\label{eq:minority-stochastic}
    \mathbb{P}_{\hat{\mu}^*} \left( X \leq \frac{n}{d^2} \;\middle|\; |I^\mathcal{D}| \leq \frac{n}{d^2}\right) \leq \Pr*{Y \leq \frac{n}{d^2}}.
  \end{equation}
  Since $\lambda \geq C_0 \frac{\log^2 d}{d^{1/2}}$, for $d$ large enough we may bound
  \[
    \mathbb{E}\left[Y\right] = \frac{\lambda}{1+\lambda} \cdot \frac{n}{2} \left(1-\frac{2}{d}\right) > \frac{n}{d}.
  \]
  Hence, by applying the Chernoff bound (\Cref{l:Chernoff}) to $Y$, we obtain
  \begin{equation}\label{eq:Chernoffapplication}
    \Pr*{Y \leq \frac{n}{d^2}} \leq \Pr*{|Y - \mathbb{E}[Y]| \geq \frac{n}{2d}} \leq 2 \exp\left(-\frac{2 n^2}{4d^2M'}\right) =  \exp\left(-\Omega\left(\frac{n}{d^2}\right)\right).
  \end{equation}
  From \eqref{eq:failure}, \eqref{eq:minority-stochastic} and \eqref{eq:Chernoffapplication} it follows that
  \begin{align*}
    \mathbb{P}_{\hat{\mu}^*}\left( \mathcal{D} \neq \mathcal{M}(I) \right) &\leq  \mathbb{P}_{\hat{\mu}^*}\left( |I^\mathcal{D}| > \frac{n}{d^2} \right)  +  \mathbb{P}_{\hat{\mu}^*}\left( X \leq \frac{n}{d^2} \;\middle|\; |I^\mathcal{D}| \leq \frac{n}{d^2} \right)  \\
   &\leq \exp\left(-\frac{n}{2d^{4}}\right) + \exp\left(-\Omega\left(\frac{n}{d^{2}}\right)\right)\\
 &\leq \exp\left(-\frac{n}{3d^{4}}\right),
  \end{align*}
  as claimed.
\end{proof}

  Recall from \Cref{sec:strategy} that for each $I \in \mathcal{I}(G)$, we define
  \begin{equation}\label{eq:omega:recall}
    \omega(I) = \lambda^{|I|}\qquad\text{and}\qquad\hat{\omega}(I) = \left(\mathbbm{1}_{\{I^\mathcal{O} \in \Omega^{\mathcal{O}}\}} + \mathbbm{1}_{\{I^\mathcal{E} \in \Omega^{\mathcal{E}}\}}\right) \lambda^{|I|},
  \end{equation}
  leading to the partition functions
  \[
    Z = \sum_{I \in \mathcal{I}(G)} \omega(I)\qquad\text{and}\qquad \hat{Z} = \sum_{I \in \mathcal{I}(G)} \hat{\omega}(I) = (1+\lambda)^{n/2} \left (\Xi^{\mathcal{O}} + \Xi^{\mathcal{E}}\right),
  \]
  and the measures
  \[
    \mu(I) = \frac{\omega(I)}{Z} \qquad\text{and}\qquad \hat{\mu}(I) = \frac{\hat{\omega}(I)}{\hat{Z}} = \hat{\mu}^*(\mathcal{O}, I) + \hat{\mu}^*(\mathcal{E}, I).
  \]

We now prove that $Z$ and $\hat{Z}$ are close to each other, that is,  the measure $\hat{\mu}$ defined via polymer models approximates the ``true'' hard-core measure $\mu$.
\begin{lemma} \label{l:approximateZ} Assume $\lambda\geq C_0 \frac{\log^2 d}{d^{1/2}}$ for some $C_0>0$. Then
\[
    \log Z - \log\left((1+ \lambda)^{n/2} \left(\Xi^\mathcal{O} + \Xi^\mathcal{E}\right)\right) = O\left(\exp\left(-\frac{n}{3d^{4}}\right)\right)
  \]
  and
  \[
    \norm{\hat{\mu} - \mu}_{TV} = O\left(\exp\left(-\frac{n}{3d^{4}}\right)\right).
  \]
\end{lemma}
\begin{proof}
  We start by comparing $\mu$ and $\hat{\mu}$. By \Cref{prop:Icaptured} and \eqref{eq:omega:recall}, we have $\hat{\omega}(I) \geq \omega(I)$ for every $I \in \mathcal{I}(G)$, with strict equality if and only if both $I^\mathcal{O} \in \Omega^{\mathcal{O}}$ and $I^\mathcal{E} \in \Omega^{\mathcal{E}}$, that is, if $I$ is captured by both the even and odd polymer models.
  Let $\mathcal{B}\subseteq \mathcal{I}(G)$ be the collection of all independent sets captured by both models. We may compute
\begin{equation}\label{eq:diff_partition}
    \hat{Z}- Z  = \sum_{I \in \mathcal{I}(G)} \left(\hat{\omega}(I) - \omega(I)\right) = \sum_{I \in \mathcal{B}} \lambda^{|I|}.
  \end{equation}
  Let $\mathcal{M}'(I)$ be the majority side for $I$, so that $\{\mathcal{M}(I), \mathcal{M}'(I)\} = \{\mathcal{O}, \mathcal{E}\}$. If $I \in \mathcal{B}$, then $I$ is counted by the polymer model for $\mathcal{M}'(I)$ as well. Therefore, recalling the definition of $\hat{\mu}^*$ in \eqref{eq:muhatstar}, we also have that 
  \begin{equation}\label{eq:majorityside}
  \frac{1}{\hat{Z}} \sum_{I \in \mathcal{B}} \lambda^{|I|} \leq \sum_{I \in \mathcal{I}(G)} \hat{\mu}^*(\mathcal{M}'(I), I) = 
  \mathbb{P}_{\hat{\mu}^*}(\mathcal{D} \neq \mathcal{M}(I)).
  \end{equation}
  Together with \eqref{eq:diff_partition}, we conclude $\hat{Z} - Z \leq \hat{Z} \cdot \mathbb{P}_{\hat{\mu}^*}(\mathcal{D} \neq \mathcal{M}(I))$. Note also that $Z \leq \hat{Z}$, since $\omega(I) \leq \hat{\omega}(I)$ pointwise.
  We may thus use Lemma~\ref{l:minoritydefect} to conclude
  \begin{equation}\label{e:Z_Zprime}
    \left(1-\exp\left(-\frac{n}{3d^{4}}\right)\right) \hat{Z} \leq Z \leq \hat{Z}.
  \end{equation}
  Using $e^{-2x} \leq 1-x$, valid for $x < 1/2$, and taking logarithms yields the first claim.
  
  Furthermore, we can bound the total variation distance between $\hat{\mu}$ and $\mu$, defined as in \eqref{eq:totalvariationdist}.
  Since $I \not\in \mathcal{B}$ implies $\hat{\omega}(I) = \omega(I)$, we have
  \begin{align}
    \norm{\hat{\mu} - \mu}_{TV} =  \sum_{\substack{I \in \mathcal{I}(G) \\ \hat{\mu}(I) > \mu(I)}} (\hat{\mu}(I) - \mu(I)) 
    &\leq \mathbb{P}_{\hat{\mu}}(I \in \mathcal{B}) + \sum_{I \not\in \mathcal{B}} \left|\frac{\omega(I)}{\hat{Z}} - \frac{\omega(I)}{Z}\right|.\label{e:TV_bound_prelim}
  \end{align}
  Using \eqref{eq:majorityside} we obtain
  \[ 
  \mathbb{P}_{\hat{\mu}}(I \in \mathcal{B}) = \sum_{I \in \mathcal{B}} \frac{2\lambda^{|I|}}{\hat{Z}} \leq 2 \cdot \mathbb{P}_{\hat{\mu}^*}(\mathcal{D} \neq \mathcal{M}(I)).
  \] 
  Applying~\eqref{e:Z_Zprime} to bound the summation on the right-hand side of~\eqref{e:TV_bound_prelim}, we conclude that
  \[
    \norm{\hat{\mu} - \mu}_{TV} \leq  \mathbb{P}_{\hat{\mu}}(I \in \mathcal{B}) + \sum_{I \not\in \mathcal{B}} \frac{\omega(I)}{Z} \cdot O\left(\exp\left(-\frac{n}{3d^{4}}\right)\right) =
    O\left(\exp\left(-\frac{n}{3d^{4}}\right)\right),
  \]
as desired.
\end{proof}

\section{Proof of \Cref{thm:general}} \label{sec:proof}

Recall that $L^{\mathcal{D}}_{k} \coloneqq \sum_{\Gamma \in \mathcal{C}^{\mathcal{D}}_{k}} w(\Gamma)$ where $\mathcal{C}^{\mathcal{D}}_{k}$ contains all clusters $\Gamma$ of size $\norm{\Gamma}=k$. Moreover, $C_0 \geq \lambda \geq C_0 \frac{\log^2 d}{d^{1/2}}$
  for some sufficiently large constant $C_0>0$ and $\lambda$ is bounded from above as $d \to \infty$.

\begin{proof}[Proof of \Cref{thm:general}]
  By definition of $L^{\mathcal{D}}_{\geq k+2}$ we have
  \[
    \Xi^{\mathcal{D}} = \exp\left(L^{\mathcal{D}}_{\leq k} + L^{\mathcal{D}}_{k+1} +L^{\mathcal{D}}_{\geq k+2}\right).
  \]
  Letting $\varepsilon' \coloneqq \log Z - \log \big((1+\lambda)^{n/2}(\Xi^{\mathcal{E}} + \Xi^{\mathcal{O}})\big)$, we have $\varepsilon' = O(\exp(-n/3d^{4}))$ by \Cref{l:approximateZ}. We may therefore write
  \begin{equation}\label{e:thm:general:first_approx}
    Z = (1+\lambda)^{n/2} \sum_{\mathcal{D} \in \{\mathcal{O}, \mathcal{E}\}} \exp\left(L^{\mathcal{D}}_{\leq k} + L^{\mathcal{D}}_{k+1} + L^{\mathcal{D}}_{\geq k+2} + \varepsilon' \right).
  \end{equation}
  We now proceed to bound the cluster expansion coefficients. We first show that for any $k\in \mathbb N$,
\begin{align}
    |L^{\mathcal{D}}_{k}|=O\left(\frac{n d^{2(k-1)} \lambda^k}{(1+\lambda)^{dk}}\right).\label{eq:boundonL}
\end{align}
    
  Let $k \in \mathbb N$ be fixed.
  Let $\Gamma \in \mathcal{C}^{\mathcal{D}}_{k}$ be a cluster of size $\norm{\Gamma}=k$.
  The vertex set $V(\Gamma) = \bigcup_{S \in \Gamma} S$ of the cluster $\Gamma$ is a $2$-linked set of size at most $k$.
  By \Cref{l:counting2linked} there are at most $n (ed^2)^{k-1}$ possible choices for $V(\Gamma)$.
  On the other hand, for every $X \subseteq V(G)$ of size at most $k$, there are only constantly many clusters $\Gamma$ such that $X=V(\Gamma)$ (recall that $k$ is fixed).
  So, in total there are $O\left(nd^{2(k-1)}\right)$ clusters of size $k$, in other words, $|\mathcal{C}^{\mathcal{D}}_{k}| =O\left(nd^{2(k-1)}\right)$.

Using $\norm{\Gamma}=\sum_{S \in \Gamma} |S|$ and $\sum_{S \in \Gamma}|N(S)| \ge |N\left(\bigcup_{S \in \Gamma} S\right)|$ we have
  \begin{align*}
    |w(\Gamma)| 
    &= |\phi(\Gamma)| \prod_{S \in \Gamma} \frac{\lambda^{|S|}}{(1+\lambda)^{|N(S)|}}
 \leq |\phi(\Gamma)| \frac{\lambda^{\norm{\Gamma}}}{(1+\lambda)^{|N\left(\bigcup_{S \in \Gamma} S\right)|}}.
  \end{align*}
Because the Ursell function of a cluster of fixed size is bounded from above by a constant (see \eqref{eq:Ursell}), and $|N\left(\bigcup_{S \in \Gamma} S\right)|\ge k(d-\Delta_2k)$, since the co-degree of $G$ is bounded by $\Delta_2$, the weight of each such cluster $\Gamma$ satisfies
  \begin{align*}
    |w(\Gamma)| & \leq |\phi(\Gamma)| \frac{\lambda^{\norm{\Gamma}}}{(1+\lambda)^{|N\left(\bigcup_{S \in \Gamma} S\right)|}} =   O\left(\frac{\lambda^k}{(1+\lambda)^{k(d-\Delta_2k)}}\right).
  \end{align*}
This implies \eqref{eq:boundonL}, since
 \[
|L^{\mathcal{D}}_{k}|=\Big|\sum_{\Gamma \in \mathcal{C}^{\mathcal{D}}_{k}} w(\Gamma)\Big| = O\left(nd^{2(k-1)} \cdot \frac{\lambda^k}{(1+\lambda)^{k(d-\Delta_2k)}}\right) = O\left(\frac{nd^{2(k-1)}\lambda^k}{(1+\lambda)^{kd}}\right) ,
\]
where the last equality holds because $\Delta_2$ and $k$ are constants and $\lambda$ is bounded. Also, by \Cref{l:Kotecky},
  \[
    |L^{\mathcal{D}}_{\geq k+2}| \leq \frac{nd^{8(k+2)}}{(1+\lambda)^{(k+2)d-\Delta_2 (k+2)^2}}.
  \]
  Observe that, since $(1+\lambda)^d = \exp(\Omega(d^{1/2} \log d)) = d^{\Omega(d^{1/2})}$, the error here is much smaller than our bound on $L^{\mathcal{D}}_{k+1}$, as we have
  \[
    |L^{\mathcal{D}}_{\geq k+2}| \leq \frac{n d^{2k} \lambda^{k+1}}{(1+\lambda)^{(k+1)d}} \cdot \frac{d^{6k+16}}{\lambda^{k+1} (1+\lambda)^{d-\Delta_2 (k+2)^2}} = o\left(\frac{n d^{2k} \lambda^{k+1}}{(1+\lambda)^{(k+1)d}}\right).
  \]
  Setting $\varepsilon^{\mathcal{D}}_{k} \coloneqq L^{\mathcal{D}}_{k+1} + L^{\mathcal{D}}_{\geq k+2} + \varepsilon'$. Since $n=\omega(d^{5})$ and $\lambda$ is bounded from above, we have $\varepsilon' = \exp(-\omega(d)) = o((1+\lambda)^{-(k+1)d})$. Therefore,
  \[
    |\varepsilon^{\mathcal{D}}_{k}| = O\left(\frac{nd^{2k} \lambda^{k+1}}{(1+\lambda)^{(k+1)d}}\right)
  \]
  and
  \begin{equation}\label{eq:pf-main-almost-final}
    Z = (1+\lambda)^{n/2} \sum_{\mathcal{D} \in \{\mathcal{O}, \mathcal{E}\}} \exp\left(\sum_{j=1}^k L^{\mathcal{D}}_{j} + \varepsilon^{\mathcal{D}}_{k}\right).
  \end{equation}
  We now compute $L_1$ explicitly. Since $G$ is $d$-regular and bipartite, its partition classes are balanced, and hence there are $n/2$ polymers $S$ of size $1$, each of which has weight
  \[
    w(S)= \frac{\lambda^{|S|}}{(1+\lambda)^{|N(S)|}}=\frac{\lambda}{(1+\lambda)^d}.
  \]
  Using that the Ursell function of the graph with a single vertex is $1$, for any $\mathcal{D} \in \{\mathcal{O}, \mathcal{E}\}$ we get
  \[
    L^{\mathcal{D}}_{1} \coloneqq \sum_{\Gamma \in \mathcal{C}^{\mathcal{D}}_1} w(\Gamma) = \frac{\lambda }{(1+\lambda)^d} \cdot \frac{n}{2}.
  \]
  Substituting this into \eqref{eq:pf-main-almost-final}, we have shown that
  \[
    Z(G, \lambda) = \left((1+\lambda)e^{\lambda/(1+\lambda)^d}\right)^{n/2}  \sum_{\mathcal{D} \in \{\mathcal{O}, \mathcal{E}\}}\exp\left(\mathbbm{1}_{k\geq 2}\cdot \sum_{j=2}^k L^{\mathcal{D}}_{j} + \varepsilon^{\mathcal{D}}_k\right)
  \]
  as claimed.
\end{proof}

\section{Expansion in product graphs: proof of \Cref{l:isoperimetry}}\label{sec:isoperimetry}

We will assume, without loss of generality, that $X \subseteq \mathcal{O}$.

\begin{proof}[Proof of \Cref{l:isoperimetry} \ref{i:isosmall}] \let\qed\relax
  In a Cartesian product graph any two vertices that differ in at least two coordinates have at most two common neighbours.
  For any two vertices that differ in exactly one coordinate, each common neighbour differs in the same coordinate. Since the size of the base graphs is by assumption at most $m$, this shows that any two vertices have at most $m$ common neighbours, proving the claim.
\end{proof}
\begin{proof}[Proof of \Cref{l:isoperimetry} \ref{i:isomedium}]\let\qed\relax
  Our proof will be by reduction to the hypercube case. Let $s$ be a constant and assume that $|X| \leq t^s$.
  Since each $H_i$ is bipartite and regular, each $H_i$ contains a perfect matching $M_i$.
  For each $i\in [t]$ let us write $m_i :=|M_i| = |H_i|/2$ and let $r :=\prod_{i=1}^t m_i$.
  The Cartesian product graph $M:=\square_{i=1}^t M_i$ is then a subgraph of $G$, and it is the union of $r$ vertex-disjoint $t$-dimensional hypercubes $G_1,G_2,\ldots, G_r$, each corresponding to the Cartesian product of some set of edges forming a transversal of the matchings $M_i$.
  Note that $\bigcup_{i=1}^r V(G_i) = V(G)$.

  Let $X_i=X \cap G_i$, so that $|X| = \sum_{i=1}^r |X_i|$.
  Since $|X_i| \leq |X| \leq t^s$ and $X \subseteq G_i \cap \mathcal{O}$, it follows from \cite[Lemma 6.2]{Ga03} that there is a constant $c > 0$ such that 
  \[
    |N_{G_i}(X_i)| \geq ct|X_i|.
  \]
  However, since $M \subseteq G$ and the $G_i$ are vertex-disjoint, it follows that
  \[
    |N_G(X)| \geq \left| \bigcup_{i=1}^r  N_{G_i}(X_i)\right| = \sum_{i=1}^r |N_{G_i}(X_i)| \geq \sum_{i=1}^r ct|X_i| = ct|X|.
  \]
\end{proof}
\begin{proof}[Proof of \Cref{l:isoperimetry} \ref{i:isolarge}]
  Assume that $|X|=\beta \frac{n}{2}$ for some $\beta \in (0, 1]$.
  Let us write $C = 2 \sqrt{2}$.
  Our proof is by induction on $t$.
  
  If $t=1$, then $G=H_1$ is a connected, bipartite, $d$-regular graph.
  In particular, $G$ has a perfect matching.

  If $X = \mathcal{O}$, then $|X| = n/2$ and $\beta = 1$. So \eqref{e:isolarge} follows from the fact that $G$ is connected such that $|N(X)|=|\mathcal{E}|=n/2$.
  Otherwise, since $G$ is connected and regular, by Hall's theorem it must be the case that $|N(X)| > |X|$ and hence
  \[
    |N(X)| \geq |X|+1 = |X| \left(1+ \frac{1}{|X|}\right) \geq |X|\left(1+ \frac{2}{\beta m}\right)\geq |X|\left(1+ \frac{C(1-\beta)}{m\sqrt{t}}\right),
  \]
  since $t = 1$, $|X| =\beta n/2 \leq \beta m /2$, $\beta(1-\beta) \leq 1/4$ and $C\leq 8$.

  Let us assume that the statement holds for Cartesian product graphs of dimension at most $t-1$.
  For each $i \in [t]$ let us write $G_i = \square_{j=1}^i H_j$, so that $G_t = G$.
  Note that each $G_i$ is bipartite, so we denote its partition classes by $\mathcal{O}_i,\mathcal{E}_i$.

  Given $v \in V(H_t)$ and $A \subseteq \mathcal{O}$ let
  \[
    A_v \coloneqq \{(x_1, ..., x_{t-1}) \in V(G_{t-1}) \;|\; (x_1, ..., x_{t-1}, v) \in A\}.
  \]
  Note that $\{ X_v \colon v \in V(H_t) \}$ is a partition of $X$.
  Let $\beta_v := 2|X_v|/|G_{t-1}|$, noting that $\beta_v \in [0,1]$, since $X_v \subseteq \mathcal{O}_{t-1}$ or $X_v \subseteq \mathcal{E}_{t-1}$.
  By the induction hypothesis,
  \[
    \left|N_{G_{t-1}}(X_v)\right| \geq |X_v| \left(1+\frac{C(1- \beta_v)}{m\sqrt{t-1}}\right).
  \]

  Suppose that there exist $v$, $w \in V(H_t)$ such that $\beta_v \geq \beta_w + \delta$, and let $(v_0, \ldots, v_k)$ be a path in $H_t$ connecting them.
  Letting $J := \{i \in [k] : \beta_{v_{i-1}} \geq \beta_{v_i}\}$, we observe that
  \[
    \frac{2}{|G_{t-1}|}\sum_{i \in J} (|X_{v_{i-1}}| - |X_{v_i}|) = \sum_{i\in J} (\beta_{v_{i-1}}-\beta_{v_i}) \geq \beta_v - \beta_w \geq \delta.
  \]
  Therefore, since $|N(X)_v| \geq |X_w|$ whenever $vw \in E(H_t)$, we have
  \begin{equation}\label{eq:neigh-gain-path}
    \sum_{i\in J} |N(X)_{v_i}| \geq \sum_{i \in J} |X_{v_{i-1}}| \geq \frac{\delta|G_{t-1}|}{2} + \sum_{i\in J} |X_{v_i}| \geq \frac{\delta |X|}{\beta m} + \sum_{i\in J} |X_{v_i}|,
  \end{equation}
  where the last inequality follows from $|G_{t-1}| = |G_t|/|H_t| \geq |G_t|/m$.
  Since $|N(X)_v| \geq |X_v|$ for every $v \in H_t$, we may add to \eqref{eq:neigh-gain-path} the contribution from vertices in $R \coloneqq V(H_t) \setminus \{v_i : i \in J\}$ to obtain
  \[
    |N(X)| = \sum_{v \in V(H_t)} |N(X)_v|  =\sum_{i \in J} |N(X)_{v_i}| + \sum_{v\in R} |N(X)_v| \geq \frac{\delta |X|}{\beta m} + \sum_{v\in V(H_t)} |X_v| = |X|\left( 1+  \frac{\delta}{\beta m}\right)
  \]
  and so taking $\delta \coloneqq C\beta(1- \beta)/\sqrt{t}$, the result would follow.

  Hence, we may assume that for all $v,w \in V(H_t)$, $|\beta_v - \beta_w| \leq \delta$.
  Applying the induction hypothesis to each $X_v$ we can conclude that
  \begin{equation}\label{e:isolarge:induction-step}
    |N(X)| \geq |X| +  \sum_v C |X_v|\frac{1-\beta_v}{m \sqrt{t-1}} = |X| +  \frac{|G_t|}{2|H_t|}\sum_v C \frac{\beta_v(1-\beta_v)}{m \sqrt{t-1}}.
  \end{equation}
  So we need to show that, for $g(x) = x(1-x)=x-x^2$,
  \begin{equation}\label{e:isolarge:induction-calculation}
    \frac{1}{|H_t|} \sum_v g(\beta_v) \geq g(\beta) \sqrt{\frac{t-1}{t}} = g(\beta) \sqrt{1 - \frac{1}{t}}.
  \end{equation}

  Since $\frac{1}{|H_t|}\sum_{v \in V(H_t)} \beta_v = \beta$ and $\max_v \beta_v - \min_v \beta_v \leq \delta$, we have \[\frac{1}{|H_t|} \sum_{v \in V(H_t)} \beta_v(1-\beta_v) = \beta(1-\beta) - \frac{1}{|H_t|}\sum_{v \in V(H_t)} \left(\beta_v-\beta\right)^2 \geq \beta(1-\beta) - \frac{\delta^2}{4},\] where the last step is an application of Popoviciu's inequality (Lemma~\ref{l:Popoviciu}).
  Substituting the value of $\delta$, we obtain
  \[
    \frac{1}{|H_t|}  \sum_v g(\beta_v) \geq g(\beta)\left(1  - \frac{C^2 \beta(1-\beta)}{4n}\right)
    \geq g(\beta)\left(1 - \frac{C^2}{16t} \right)
    \geq g(\beta) \left(1- \frac{1}{2t}\right),
  \]
  using the fact that $\beta(1-\beta) \leq \frac{1}{4}$ and since $C^2 \leq 8$.
  Since $\sqrt{1-x} \leq 1 - \frac{x}{2}$ for $x \in [0,1]$, this implies~\eqref{e:isolarge:induction-calculation}, proving~\eqref{e:isolarge:induction-step} and finishing the induction step.
\end{proof}

\section{Cartesian power of complete bipartite graphs: proof of \Cref{thm:bipartiteproduct}} \label{sec:productcomplete}

Throughout this section we will fix some $s \in \mathbb{N}$ and consider the $t$-th Cartesian power $G=\car_{i=1}^t K_{s, s}$ of a complete bipartite graph $K_{s,s}$ as $t \to \infty$.
Note that $G$ is $d$-regular with $d=st$, and it has $n=(2s)^t$ vertices.

By \Cref{thm:general} with $k=2$ we have
\begin{equation}  \label{eq:bipartiteproduct}
  Z(G, \lambda) = 2(1+\lambda)^{n/2} \exp\left(\frac{n \lambda }{2(1+\lambda)^d} + L_2 +\varepsilon_2\right),
\end{equation}
with $|L_2|=O\left(\frac{n d^2 \lambda^{2} }{(1+\lambda)^{2d}}\right)$ and $|\varepsilon_2| =O\left(\frac{n d^{4} \lambda^{3} }{(1+\lambda)^{3d}}\right)$. Note that since $K_{s,s}$ is symmetric we will compute the coefficients for the even side (suppressing the $\mathcal{E}$ superscript) and multiply by a factor of two.

Taking $\lambda=1$ in \eqref{eq:bipartiteproduct} leads to the following bound on $i(G)$
\begin{equation*}
  i(G)=Z(G, 1) =2^{1 + (2s)^t/2} \exp\left(\frac{(2s)^t}{2^{st+1}} + O\left(\frac{(2s)^t (st)^2}{2^{2st}}\right)\right).
\end{equation*}

Next we refine the above bound on $i(G)$ by calculating the next term $L_2$.
Note that $L_2$ is a sum over all clusters $\Gamma$ with $\norm{\Gamma} = \sum_{S \in \Gamma} |S|=2$.
There are several possibilities for the structure of a cluster of size $2$, see \Cref{fig:powerKmm}, giving the following cases:

\begin{enumerate}[$(i)$]
\item \label{twice-same-polymer} $\Gamma$ contains twice the same polymer $S$ of size $1$;
\item \label{one-polymer-size-two} $\Gamma$ contains one polymer of size $2$;
\item \label{two-different-polymers} $\Gamma$ contains two different polymers of size $1$.
\end{enumerate}
In each case we need to count the number of clusters of the given type and evaluate their weight.

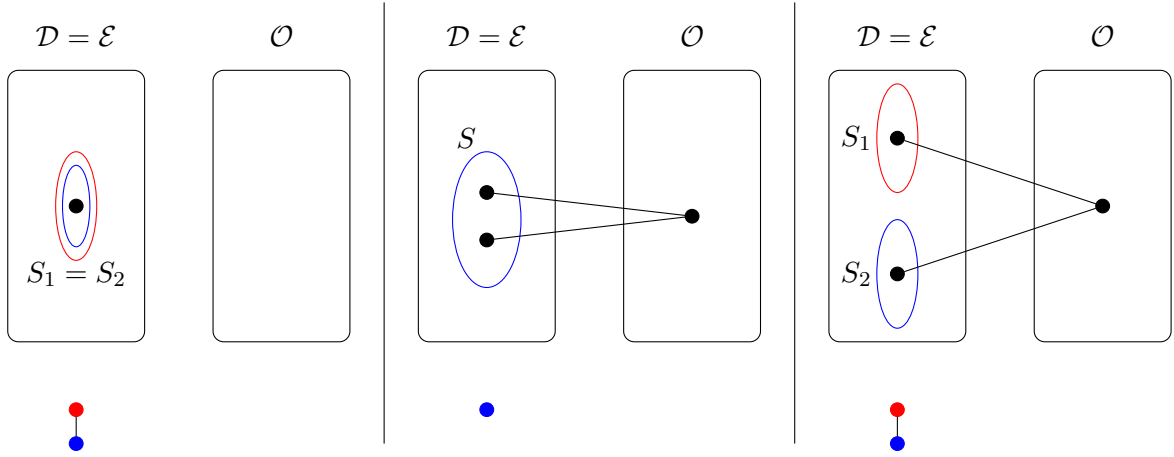
\begin{figure}
  \centering
  \begin{tikzpicture}[scale=0.9]
    \path[draw=black] (4.5,-1.5) -- (4.5,5);
    \path[draw=black] (10.5,-1.5) -- (10.5,5);


    \path[draw=black,rounded corners] (-1,0) -- (-1,4) -- (1,4) -- (1,0) -- cycle;
    \path[draw=black,rounded corners] (2,0) -- (2,4) -- (4,4) -- (4,0) -- cycle;
    \node[] at (0,4.5) {$\mathcal{D}=\mathcal{E}$};
    \node[] at (3,4.5) {$\mathcal{O}$};

    \draw[blue] (0, 2) ellipse (0.2cm and 0.6cm);
    \draw[red] (0, 2) ellipse (0.3cm and 0.8cm);
    \node (Y1) at (0, 2) [circle,draw, fill, scale=0.5] {};
    \node[] at (0,1) {$S_1=S_2$};

    \node (D) at (0, -1) [circle,draw=red, fill=red, scale=0.5] {};
    \node (E) at (0, -1.5) [circle,draw=blue, fill=blue, scale=0.5] {};
    \path[draw=black] (D) -- (E);
    
    
    \path[draw=black,rounded corners] (5,0) -- (5,4) -- (7,4) -- (7,0) -- cycle;
    \path[draw=black,rounded corners] (8,0) -- (8,4) -- (10,4) -- (10,0) -- cycle;
    \node[] at (6,4.5) {$\mathcal{D}=\mathcal{E}$};
    \node[] at (9,4.5) {$\mathcal{O}$};

    \draw[blue] (6, 1.8) ellipse (0.5cm and 1cm);
    \node (Y1) at (6, 1.5) [circle,draw, fill, scale=0.5] {};
    \node (Y2) at (6, 2.2) [circle,draw, fill, scale=0.5] {};
    \node (Y3) at (9, 1.85) [circle,draw, fill, scale=0.5] {};
    \path[draw=black] (Y1) -- (Y3) -- (Y2);
    \node[] at (5.7,3) {$S$};

    \node (A) at (6, -1) [circle,draw=blue, fill=blue, scale=0.5] {};
    

    \path[draw=black,rounded corners] (11,0) -- (11,4) -- (13,4) -- (13,0) -- cycle;
    \path[draw=black,rounded corners] (14,0) -- (14,4) -- (16,4) -- (16,0) -- cycle;
    \node[] at (12,4.5) {$\mathcal{D}=\mathcal{E}$};
    \node[] at (15,4.5) {$\mathcal{O}$};

    \draw[blue] (12, 1) ellipse (0.3cm and 0.8cm);
    \draw[red] (12, 3) ellipse (0.3cm and 0.8cm);
    \node (Y1) at (12, 1) [circle,draw, fill, scale=0.5] {};
    \node (Y2) at (12, 3) [circle,draw, fill, scale=0.5] {};
    \node (Y3) at (15, 2) [circle,draw, fill, scale=0.5] {};
    \path[draw=black] (Y1) -- (Y3) -- (Y2);
    \node[] at (11.4,3) {$S_1$};
    \node[] at (11.4,1) {$S_2$};

    \node (B) at (12, -1) [circle,draw=red, fill=red, scale=0.5] {};
    \node (C) at (12, -1.5) [circle,draw=blue, fill=blue, scale=0.5] {};
    \path[draw=black] (B) -- (C);
  \end{tikzpicture}
  \caption{From left to right the three possibilities for the structure of a cluster of size $2$ and the corresponding incompatibility graphs.}
  \label{fig:powerKmm}
\end{figure}

To motivate and illustrate the main ideas of the more general algorithm presented in \Cref{sec:computation}, we will count clusters by constructing a list of `prototypical' examples of clusters of size two with their weights. Then we will count the ways to embed these sets into the graph $G$.
Note that each cluster of size two spans a $2$-linked set of vertices of size at most two.
Thus, we need a (minimal) list of `compressed' $2$-linked sets $S_1, ..., S_{\ell}$ of size at most two such that for every $2$-linked set $T$ of size at most two there is an automorphism of $G$ that maps $T$ to some set $S_i$ from our list preserving the neighbourhood structure.

For each compressed $2$-linked set $S_1, ..., S_{\ell}$, we build all possible clusters $\Gamma_{i,1},..., \Gamma_{i, m(i)}$ such that $\bigcup_{S \in \Gamma_{i,j}} S = S_i$ for $i \in [\ell]$ and $j \in [m(i)]$. If $S_i$ admits $\alpha_i$ rooted embeddings (i.e., embeddings with a distinguished vertex) into $G$, then $\Gamma_{i, j}$ contributes $\frac{\alpha_i}{|S_i|} w(\Gamma_{i, j}) \phi(\Gamma_{i, j})$ to $L_2$, where $\phi$ denotes the Ursell function defined in~\eqref{eq:Ursell}. Thus, we have
\begin{equation}\label{eq:l2-compute}
L_2= \sum_{i=1}^{\ell} \frac{\alpha_i}{|S_i|} \sum_{j=1}^{m(i)} w(\Gamma_{i, j}) \phi(\Gamma_{i, j}).
\end{equation}
The embeddings we consider will respect the bipartition, mapping even (resp.\ odd) vertices to even (resp.\ odd) vertices.
For $G=\car_{i=1}^t K_{s, s}$, we label the vertices in $K_{s,s}$ with $0, 1, ..., 2s-1$ such that one bipartition class contains all even indices. Then a list of compressed $2$-linked sets in $G$ is given by the following collection:
\[
S_1=\{(0,...,0)\}, \qquad S_2=\{(0, ..., 0), (1, 1, 0, ..., 0) \}, \qquad S_3=\{(0, ..., 0), (2, 0, ..., 0)\}.
\]

The way we generate these sets is by choosing a root (in this case the vertex $(0,...,0)$) and then computing all possibilities to build a $2$-linked set containing the root and changing precisely a prefix of the coordinates (to a finite number of values, avoiding `gaps' in the used even/odd values in each coordinate). We will describe this in more detail in the next section.

There are $\alpha_1=\frac{n}{2}$ ways to embed $S_1$ into $G$ because there are $\frac{n}{2}$ ways to map the root. To embed $S_2$, we must pick where to map the root, as well as two out of $t$ coordinates to change and in each of these coordinates one out of $s$ values to change this coordinate to. Since these are odd coordinates, there are $\alpha_2= \frac{n}{2} \cdot \binom{t}{2} \cdot s^2$ ways to embed $S_2$ into $G$.
To embed $S_3$ we must embed the root, choose one out of $t$ coordinates to change, and one out of $s-1$ values that we can change this coordinate to (recall that we need to change to a vertex with the same parity as the corresponding vertex). Thus, there are $\alpha_3= \frac{n}{2} \cdot t \cdot (s-1)$ ways to embed $S_3$ into $G$. 

Now we generate the clusters that have these $2$-linked sets as underlying vertex sets. A simple calculation gives that the Ursell function $\phi$ as given in \eqref{eq:Ursell} of a single vertex is $1$ while the Ursell function $\phi$ of a single edge is $-\frac{1}{2}$.
In order to get a cluster from $S_1$ the only option is to pick $S_1$ twice. So,
\[
\Gamma_{1,1} = (S_1, S_1) \qquad \text{ with } \qquad w(\Gamma_{1,1})= -\frac{1}{2} \cdot \left( \frac{1}{2}\right)^{2d},
\]
since a vertex has $d$ neighbours and the incompatibility graph of this cluster is an edge. Note that $\Gamma_{1,1}$ represents Case~\ref{twice-same-polymer}.

In order to get a cluster from $S_2$, there are two options. Either we pick one polymer of size two or two different polymers of size one. In the first case we get
\[
\Gamma_{2,1} = (S_2) \qquad \text{ with } \qquad w(\Gamma_{2,1})= \left( \frac{1}{2}\right)^{2d-2},
\]
since $|N(S_2)|=2d-2$. The incompatibility graph of this cluster is a single vertex, as in Case~\ref{one-polymer-size-two}.
We can also form the cluster corresponding to Case~\ref{two-different-polymers}, which is
\[
\Gamma_{2,2} = (\{(0,...,0)\},\{(1,1,0,...,0)\}) \qquad \text{ with } \qquad w(\Gamma_{2,2})= - \frac{1}{2} \left( \frac{1}{2}\right)^{2d},
\]
since for the weight we multiply the weights of the single polymers each consisting of one vertex with $d$ neighbours and the incompatibility graph of this cluster is an edge. Recall that clusters are ordered tuples of polymers, so formally we also get $\Gamma_{2,3} = (\{(1,1,0,...,0)\}, \{(0,...,0)\})$. But this cluster is very similar to $\Gamma_{2,2}$, so we will take it into account by multiplying by a factor of $2$.

Similarly, for $S_3$ there are also two options. In the first case we get
\[
\Gamma_{3,1} = (S_3) \qquad \text{ with } \qquad w(\Gamma_{3,1})= \left( \frac{1}{2}\right)^{2d-s},
\]
since $|N(S_3)|=2d-s$. The incompatibility graph of this cluster is a single vertex.
In the second case we get
\[
\Gamma_{3,2} = (\{(0,...,0)\},\{(2,0,...,0)\}) \qquad \text{ with } \qquad w(\Gamma_{3,2})= - \frac{1}{2} \cdot \left( \frac{1}{2}\right)^{2d},
\]
since the incompatibility graph of this cluster is an edge.
As before, formally we also get a cluster $\Gamma_{3,3}=(\{(2,0,...,0)\}, \{(0,...,0)\})$ which we take into account by multiplying by a factor of $2$ when we count $\Gamma_{3,2}$.

Substituting the possible clusters (in the order as listed) with their weights and the number of ways to generate them into \eqref{eq:l2-compute} we obtain
\begin{equation} \label{eq:L2start}
    \begin{split}
         L_2 &=\phantom{+} w(\Gamma_{1,1}) \cdot \frac{n}{2} \\
    &\phantom{+} + w(\Gamma_{2,1}) \cdot \frac{ns^2}{4}\binom{t}{2}+ 2 \cdot w(\Gamma_{2,2}) \cdot \frac{ns^2}{4}\binom{t}{2} \\
    &\phantom{+} + w(\Gamma_{3,1}) \cdot \frac{nt(s-1)}{4} + 2 \cdot w(\Gamma_{3,2}) \cdot \frac{nt(s-1)}{4}.
    \end{split}
  \end{equation}
  In the calculation above, we already divided by $|S_i|$ to compensate for the fact that the vertex $(0,\ldots,0)$ was distinguished in the embeddings, as previously discussed.   
Plugging the weights as above into \eqref{eq:L2start} and simplifying the terms this becomes
\begin{equation} \label{eq:L2weights}
    \begin{split}
        L_2 &= -\left( \frac{1}{2}\right)^{2d+2} n \\
    &\quad + \left( \frac{1}{2}\right)^{2d}\cdot ns^2\binom{t}{2} - \left( \frac{1}{2}\right)^{2d+2} \cdot ns^2\binom{t}{2} \\
    &\quad +  \left( \frac{1}{2}\right)^{2d-s+2}\cdot nt(s-1) - \left( \frac{1}{2}\right)^{2d+2} \cdot nt(s-1).
    \end{split}
\end{equation}

Recalling that $n=(2s)^t$ and $d=st$, \eqref{eq:L2weights} can be simplified to
\begin{equation} \label{eq:L2final}
    \begin{split}
        L_2 &= 3 \binom{t}{2} \frac{(2s)^t s^2}{2^{2st+2}} + (2^s- 1) \frac{(2s)^t(s-1) t}{2^{2st+2}} - \frac{(2s)^t}{2^{2st+2}}\\
            &=\frac{(2s)^t}{2^{2st+2}} \left( 3\binom{t}{2} + (2^s-1)(s-1)t - 1\right).
    \end{split}
\end{equation}

Plugging \eqref{eq:L2final} into \eqref{eq:bipartiteproduct} and using that $\varepsilon_2=O\left((2s)^t (st)^4 \cdot 2^{-3st} \right)$ establishes \Cref{thm:bipartiteproduct}.

\section{Computing the partition function} \label{sec:computation}

The procedure from \Cref{sec:productcomplete} can be turned into an algorithm that computes $L_j$ for arbitrary $j \in \mathbb{N}$ for some Cartesian powers of a fixed vertex-transitive graph $H$. It can also be modified to deal with certain Cartesian powers of graphs which may depend on a parameter. Some sufficient conditions on the base graphs to allow for such an algorithm are discussed. We use the case $G = \car_{i=1}^t K_{s,s}$ to illustrate the required changes.

\subsection{Compressed sets} To compute $L_j$, the first step will be to compute the family of all $2$-linked subsets of size at most $j$ containing a distinguished vertex.
To turn this into a finite problem independent of the dimension $t$, we may follow the idea described in~\cite{JePe2020}. 
Given a defect side, which we may assume without loss of generality to be $\mathcal{E}$, we will fix a vertex $r \in \mathcal{E}$ called the root. We say a $2$-linked set $S$ containing the root is \emph{$r$-compressed} if the set of active coordinates, defined by
\[ A(S) \coloneqq \{i \in [t] : v_i \neq r_i\text{ for some }v \in S\}, \]
consists of the first $|A(S)|$ positive integers. 
Note that, since $S$ is $2$-linked, $|A(S)| \leq 2|S|$.
When the choice of root $r$ is clear, we will omit it from notation.

If the base graph has bounded size, as in~\cite{JePe2020}, one may generate all compressed sets of size up to $j$ in $e^{O(j \log j)}$ time. To do so, first notice that $A(S) \subseteq [2j]$ for every compressed set of size up to $j$, since in a spanning tree for $G^2[S]$ adjacent vertices differ in at most two coordinates. Let $\mathcal{L}_k$ denote the family of (not necessarily $r$-compressed) $2$-linked sets $S$ containing $r$ such that $A(S) \subseteq [2j]$ and $|S| = k$. By the $2$-linked condition, every element in $\mathcal{L}_{k+1}$ can be obtained by choosing one element of $\mathcal{L}_k$ and changing at most two coordinates of it. Therefore, $|\mathcal{L}_{k+1}| \leq |\mathcal{L}_k| \cdot \binom{2j}{2} v(H)^2$.
By induction, $|\mathcal{L}_j| = e^{O(j \log j)}$ since $\mathcal{L}_1 = \{\{r\}\}$. 

The above procedure is constructive, and one may check whether each element of $\bigcup_{k=1}^j \mathcal{L}_k$ is $r$-compressed in polynomial time per set, yielding the desired list. If implemented as described, however, the method above has a significant drawback: A single element of $\mathcal{L}_{k+1}$ will be obtainable from elements of $\mathcal{L}_k$ in multiple ways, and one must keep a list of the generated sets in memory to avoid overcounting, leading to exponential space usage. Algorithm~\ref{a:generatecompressed}, $\textsc{GenerateCompressed}$, is such that \textsc{GenerateCompressed}($N^2(r)$, $\{r\} \cup N^2(r)$, $\{r\}$) (and its recursive subcalls) outputs every compressed set of size at most $j$ \textit{exactly once}, using $e^{O(j \log j)}$ time and polynomial space.

\begin{center}
\begin{minipage}{.52\linewidth}
\begin{algorithm}[H]
  \caption{\textsc{GenerateCompressed}($Q$, $R$, $S$)}\label{a:generatecompressed}
  \begin{algorithmic}[1]
    \If{$S$ is compressed \textbf{and} $|S| \leq j$}
      \State{\textbf{output} $S$}
    \EndIf
    
    \If{$|S| < j$}\label{a:if-cond}
      \State let $(v_1, \ldots, v_k) = Q$
      \For{$i = 1, \ldots, k$}
        \State $Q' := (Q \setminus \{v_1, \ldots, v_i\}) \cup (N^2(v_i) \setminus R)$
        \State $R' := R \cup N^2(v_i)$
        \State $S' := S \cup \{v_i\}$
        \State \textsc{GenerateCompressed}($Q'$, $R'$, $S'$)
        \EndFor
      \EndIf
  \end{algorithmic} 
\end{algorithm}
\end{minipage}
\end{center}

As argued before, every compressed set $S$ of size up to $j$ satisfies $A(S) \subseteq [2j]$, so we may assume the algorithm only stores in memory the first $2j$ coordinates of every considered vertex, the remaining coordinates being implicit (and equal to the root). In particular, for $v \in V(G)$, the notation $N^2(v)$ refers to vertices $w$ at distance $2$ from $v$ such that $v_i = w_i$ for $i > 2j$.

Informally, vertices in $Q$ are pending a decision as to whether they will be added into $S$, and vertices of $R \setminus (S \cup Q)$ were in $Q$ at some point but were rejected.
More precisely, a recursive subcall of the form \textsc{GenerateCompressed}($Q$, $R$, $S$) outputs every compressed set $S'$ of size up to $j$ such that $S \subseteq S'$ and $S' \cap R \subseteq S \cup Q$.
The invariant $Q \subseteq R \subseteq S \cup N(N(S))$ is also maintained, and since only elements of $Q$ are added to $S$, this implies that the sets $S$ in the recursive calls are always $2$-linked.
Since every recursive call adds an element to $S$, the recursion depth is bounded by $j$.

It remains to estimate the complexity of the algorithm.
Note that an element added to $R$ in a recursive call is never removed in a subcall.
This ensures that a vertex added to $Q$ in a recursive call is never added again in a subcall, even if it was removed from $Q$ in the meantime.
Using this, it can be shown that each recursive call corresponds to a different value for $S$.
Therefore, the algorithm makes $|\bigcup_{k=1}^j \mathcal{L}_k|$ recursive calls, and each call spends time polynomial in $j$ since
\[
|Q \cup R \cup S| \leq |S \cup N(N(S))| \leq (1+(2j)^2\Delta(H)^2) |S| = O(j^3).
\]
This shows that the algorithm runs in time $e^{O(j \log j)}$, and uses polynomial space.


In practice, it is useful to augment the condition in line \ref{a:if-cond} so that sets $S$ which are ``too far'' from being compressed are discarded early. For example, if $\big|[\max A(S)] \setminus A(S)\big| > 2(j - |S|)$, then no $2$-linked set $S'$ with $|S'| \leq j$ and $S \subseteq S'$ is such that $A(S')$ is a prefix of the positive integers, implying there is no need to consider the (potentially many) ways of extending $S$.

\subsubsection{Modifications when $H = K_{s,s}$} When the base graphs depend on a parameter $s$, it is important that the number of compressed sets is independent of this parameter\footnote{This is not always a problem. For example, when $H = C_{2s}$ for $v \in V(H)$, the set $\{w \in V(H) : d(v, w) \leq 2j\}$ has at most $4j$ elements. Therefore, the definition of compressed set does not need adaptation in this case.}. We will illustrate how to modify the definition of being compressed to handle the case $K_{s,s}$. As before, we will assume that the defect side is $\mathcal{E}$.
For a set $S$ to be compressed we will require that, for each coordinate $i \in [A(S)]$, the sets \[ O_i(S) \coloneqq \{v_i : v \in S\} \cap \mathcal{O} \qquad\text{and}\qquad E_i(S) \coloneqq \{v_i : v \in S\} \cap \mathcal{E} \] are prefixes of the odd and even non-negative integers, respectively. We may obtain many embeddings of a given compressed set $S$ into $\car_{i=1}^t H$ by choosing
\begin{itemize}
 \item $r' \in \mathcal{E}$, the image of $r = (0, \ldots, 0)$,
 \item an increasing sequence $1 \leq a_1 < \cdots < a_{|A(S)|} \leq t$ of active coordinates,
 \item for each $i\in [t]$, one sequence $0 < o_{i, 1} < \cdots < o_{i, |O_i(S)|} \leq 2s-1$ of odd integers,
 \item for each $i\in [t]$, one sequence $0 = e_{i, 0} < \cdots < e_{i, |E_i(S)|-1} \leq 2s-2$ of even integers.
\end{itemize}
Every such choice defines an embedding $\psi$ by
\[ 
\psi(v_1, \ldots, v_t)_j \coloneqq \begin{cases}
    r'_j &\text{if }j \neq a_i\text{ for every }i \in A(S), \\
    r'_j + o_{j,k} &\text{if }j = a_i\text{ and }v_i = 2k-1, \\
    r'_j + e_{j,k} &\text{if }j = a_i\text{ and }v_i = 2k,
\end{cases}
\]
where addition is understood to be modulo $2s$.
From the definition, it follows that the family $\Psi(S)$ of all embeddings with a distinguished vertex has size
\[
|\Psi(S, r)| = \frac{n}{2} \binom{t}{|A(S)|} \prod_{i=1}^{|A(S)|} \binom{s}{|O_i(S)|} \binom{s-1}{|E_i(S)|-1}.
\]
The expression for the embeddings themselves is not computationally relevant, but, to count sets exactly once, it is important that, for a given $r$, the family $\{\Psi(S, r) : S\text{ is compressed}\}$ partitions the family of $2$-linked sets of size at most $j$ with a distinguished root.
Only the quantity $|\Psi(S, r)|$ will be used. Since clusters do not have a distinguished root, we will later divide by $|S|$.

A key property we require from these embeddings is that, for every $\psi \in \Psi(S, r)$, the number of ways of splitting $\psi(S)$ into polymers, and the weight of the corresponding clusters, is the same. Therefore, to process the sets in $\Psi(S, r)$, we may consider clusters supported on $S$ and simply multiply the answer by $|\Psi(S, r)|/|S|$ in the end.

To generate compressed sets for $\car_{i=1}^t K_{s,s}$ using the $\textsc{GenerateCompressed}$ procedure, we may consider the root to be $(0, \ldots, 0)$, and we must interpret $N^2(v)$ with regard to the $(2j)$-th Cartesian power of $K_{j,j}$, i.e., $\car_{i=1}^{2j} K_{j,j}$, viewed as a subgraph of $\car_{i=1}^{t} K_{s,s}$ in the natural way by taking $V(K_{j,j}) = [2j]$, $V(K_{s,s}) = [2s]$ and appending enough zeros to account for the missing coordinates. One may check that, for $j \geq 1$, every compressed set of size at most $j$ is contained in this subgraph.

Before proceeding to the next step, we remark that it can be profitable to exploit more symmetries when defining compressed sets. In the above case, further requiring that $|O_i(S)| + |E_i(S)|$ is non-increasing in $i$, for example, reduces the number of compressed sets $S$ significantly and improves the execution time of the algorithm. We note that, for both definitions of compressed, the sets $S_1$, $S_2$ and $S_3$ in \Cref{sec:productcomplete} are the only $r$-compressed sets of size at most $2$ containing $r = (0,\ldots, 0)$.

\subsection{Clusters and the Ursell function}

Given a set $S$ of size at most $j$, for which we may assume $|A(S)| \subseteq [2j]$, our goal is to generate all clusters $\Gamma$ such that $\norm{\Gamma} = j$ and $\bigcup_{S' \in \Gamma} S' = S$. Since both the weight of a cluster and its Ursell function do not depend on the polymer order, we may deal with multisets and multiply by the appropriate multinomial coefficient.

We start by observing that there are at most $\exp^{O(j \log j)}$ ways to cover $S$ by sets whose sums of sizes is $j$. Indeed, this number is at most
\[ \sum_{k=1}^j \sum_{\substack{j_1 + \ldots + j_k = j \\ j_1, \ldots, j_k \geq 1}} \binom{|S|}{j_1} \cdots \binom{|S|}{j_k} \leq 2^{j-1} \cdot |S|^{j} = e^{O(j \log j)}. \]
A backtracking procedure which considers each subset of $S$ (checking if they are $2$-linked first) generates these covers. Indeed, as long as we immediately discard partial solutions $S'_1, \ldots, S'_\ell$ such that $|S \setminus \bigcup_{i=1}^{\ell} S'_i| + \sum_{i=1}^{\ell} |S'_i|  > j$, and process singletons only at the end, this procedure never hits a ``dead end'' and generates all clusters in time $e^{O(j \log j)}$ and polynomial space. 
If space is not a constraint, we found it useful to precompute the candidate polymers (subsets of $S$ which are $2$-linked), as well as the weight for each such set, and use only such sets in the backtracking procedure.

Recall that each cluster $\Gamma = (S_1, \ldots, S_k)$ with $\norm{\Gamma} = j$ contributes \[ 
\phi(\Gamma) \cdot \prod_{i=1}^k \frac{\lambda^{|S_i|}}{(1+\lambda)^{|N(S_i)|}} 
\] to the value of $L_j$, where $\phi$ is the Ursell function of the cluster. Note that in $G=\prod_{i=1}^t K_{s,s}$ the degree of the base graph depends on a parameter $s$. Nonetheless, the explicit structure of $K_{s,s}$ allows computing neighbourhood intersections efficiently. To do this, note that any two vertices which differ in one coordinate have exactly $s$ common neighbours.

Finally, we briefly discuss the computation of the Ursell function $\phi$. If $H = H(\Gamma)$ is the incompatibility graph of $\Gamma$, then $\phi(\Gamma)$ equals $\frac{(-1)^{v(H)-1}}{v(H)!} \cdot T_H(1, 0)$, where $T_H$ is the Tutte polynomial of $H$. We may specialise the deletion-contraction formula to obtain that, for connected $H$,
\begin{equation*}
    T_H(1, 0) = \begin{cases}
        1&\text{ if $H = K_1$},\\
        T_{H / e}(1,0)&\text{ if $e$ is a bridge},\\
        T_{H / e}(1,0) + T_{H \setminus e}(1,0)&\text{ otherwise},
    \end{cases}
\end{equation*}
where $H / e$ denotes the graph obtained by $H$ by contracting the edge $e$. Due to the fact that $T_H(1,0) = 0$ whenever $H$ contains a loop, we may perform the contraction procedure so that it deduplicates parallel edges. In this way, all graphs considered in recursive applications remain simple. 

As observed in~\cite{SeImTa95}, if we consider any sequence of deletions and contractions obtained in reducing $H$ to $K_1$, the contracted edges uniquely determine a spanning tree of $H$. Therefore, the number of leaves in the recursion tree for the computation of $T_H$ is at most the number of spanning trees of $H$, which is at most $j^{j-2}$ for a $j$-vertex graph  by Cayley's formula. Therefore, deletion-contraction may be performed in $e^{O(j \log j)}$ time and polynomial space. If extra space is available, caching results (including subproblems and taking isomorphisms into account) is very helpful in practice. See~\cite{HaPeRo10} for other practical aspects of computing Tutte polynomials.

An alternate procedure with similar complexity may be obtained by observing that, for any $v \in V(H)$, $T_H(1, 0)$ counts the number of acyclic orientations of $H$ with unique sink $v$, as originally shown by Greene and Zaslavsky~\cite{GrZa83} (see also~\cite[Theorem X.8]{Bollobas}). Trivially there are at most $j! = \exp^{O(j \log j)}$ such orientations, and it is possible to enumerate them at polynomial cost per orientation using polynomial space~\cite{CoGrMaRi18}.

We note that computing the Ursell function is not the bottleneck, both in theory and in practice. Indeed, there is an $e^{O(j)}$-time algorithm to compute Tutte polynomials~\cite{BjHuKaKo2008} of a $j$-vertex graph which uses polynomial space, and the authors of~\cite{BjHuKaKo2008} observed that it outperforms the deletion-contraction strategy in the worst case when $j \geq 13$. In the opposite direction, it was shown in~\cite{DeHuMaTaWa14} that under a weakening of the Exponential Time Hypothesis, namely assuming that there exists a constant $c > 0$ such that no deterministic algorithm can compute the number of solutions to a \textsc{3-Sat} problem with $n$ variables in time $e^{cn}$, computing the Ursell function of a $j$-vertex graph requires time $e^{\Omega(j / (\log j)^2)}$ in the worst case.

\subsection{Structural assumptions to allow efficient computation}

Let us discuss what base graphs allow for such an algorithm
and which properties are used.
The algorithm as presented computes first a list of compressed sets and the number of ways to embed them into the graph $G$. Then it computes the weight of each compressed set, which boils down to computing the neighbourhood of the compressed set.

In order to be able to compute efficiently, a finite number of compressed sets needs to represent every possible $2$-linked set. Let us make that more precise using an equivalence relation.
We say two rooted $2$-linked sets $(S_1, r_1) \sim (S_2, r_2)$ (where $r_i$ is the root of $S_i$, i.e., a distinguished vertex in $S_i$) are equivalent if there is a graph isomorphism $\varphi$ between $G[S_1 \cup N(S_1)]$ and $G[S_2 \cup N(S_2)]$ which satisfies $\varphi(S_1)=S_2$ and more specifically $\varphi(r_1)=r_2$.
A list of compressed sets then corresponds to a set of representatives of every equivalence class of this relation.
Formally, we thus require this equivalence relation to have a finite number of equivalence classes where a representative from each class is efficiently computable. 
We also need to efficiently compute the number of sets in an equivalence class since this corresponds to the number of embeddings into $G$.
In practice, many families of combinatorially-relevant graphs exhibit `local self-similarity', with smaller graphs in the family having `canonical' embeddings into bigger ones. In many such cases, the generation of compressed sets may be done with regards to a constant-size graph, as in the $\car_{i=1}^t K_{s,s}$ case described above.

Finally, one needs to compute the weight of clusters corresponding to the computation of neighbourhood sizes. If the base graphs have constant maximum degree, then computing $|N(S_i)|$ for each polymer may be done by simply listing neighbours whose active set of coordinates is a subset of $A(S_i)$. Alternatively, if we let $N_i(v) := \{w \in N(v): w_j = v_j\text{ for }j \neq i\}$, we may observe that, for a polymer $S$, each of the terms in the sum $\sum_{i \in A(S)} \bigl|\bigcup_{v \in S} N_i(v)\bigr|$ may be computed in time $\exp(O(k))$ using inclusion-exclusion. This sum overestimates $|N(S)|$, but it only overcounts vertices $w$ for which there exist $v_1, v_2 \in S$ and distinct $i, j \in A(S)$ such that $w \in N_i(v_1) \cap N_j(v_2)$. Since the $O(k^4)$ sets of the form $N_i(v_1) \cap N_j(v_2)$ each have cardinality $0$ or $1$ in a Cartesian product graph, we may correct the overcount in polynomial time. This strategy is useful when we know how to compute neighbourhood intersection sizes in the base graphs efficiently.

\subsection{Implementation} To test the above ideas, we have computed the first six coefficients of the cluster expansion for $G = \car_{i=1}^t K_{s,s}$, which can be found in Appendix~\ref{app:coefficients}. The source code is available at \url{https://github.com/collares/cluster-coefficients}. The approximate computing times for general $\lambda$ are listed in Figure~\ref{f:computing}.
\begin{figure}[h]
\centering
\begin{tabular}{rrrl}\toprule
    Coefficient & Compressed sets & Covers (multisets) & Approximate times \\\midrule
    $L_2$ & \num{3} & \num{5} & \\ 
    $L_3$ & \num{37} & \num{151} & 5 milliseconds \\ 
    $L_4$ & \num{1712} & \num{14954} & 0.4 seconds \\ 
    $L_5$ & \num{187082} & \num{3338633} & 2.2 minutes \\ 
    $L_6$ & \num{36331337} & \num{1312882496} & 1.4 days \\\bottomrule 
\end{tabular}
\caption{Computation times for first coefficients}\label{f:computing}
\end{figure}
 
The $d$-dimensional hypercube can be obtained by setting $s = 1$ and $t = d$, or by setting $s = 2$ and $t = d/2$. Therefore, we may extend the calculation of~\cite{JePe2020} to obtain the following extra coefficients for the hypercube.

\begingroup
\allowdisplaybreaks
\begin{align*}
    L_4 \cdot 2^{3 \, d} &= \frac{9}{8} \, d^{6} - \frac{21}{4} \, d^{5} + \frac{509}{64} \, d^{4} - \frac{2201}{96} \, d^{3} + \frac{3683}{64} \, d^{2} - \frac{3691}{96} \, d - \frac{1}{8}, \\
    L_5 \cdot 2^{4 \, d} &= \frac{675}{256} \, d^{8} - \frac{1125}{64} \, d^{7} + \frac{6767}{128} \, d^{6} - \frac{15593}{120} \, d^{5} - \frac{234607}{768} \, d^{4} \\
    &\quad + \frac{194825}{64} \, d^{3} - \frac{1134821}{192} \, d^{2} + \frac{783331}{240} \, d + \frac{1}{10}, \\
    L_6 \cdot 2^{5 \, d} &= \frac{2187}{320} \, d^{10} - \frac{3807}{64} \, d^{9} + \frac{16749}{64} \, d^{8} - \frac{10777}{16} \, d^{7} - \frac{782797}{480} \, d^{6} + \frac{702581}{480} \, d^{5} + \frac{3515827}{32} \, d^{4} \\
    &\quad - \frac{30903735}{64} \, d^{3} + \frac{174959297}{240} \, d^{2} - \frac{85287403}{240} \, d - \frac{1}{12}.
\end{align*}
\endgroup

The coefficients obtained for $G=\car_{i=1}^t K_{s,s}$ can be found in \Cref{app:coefficients}.

\section{Discussion} \label{sec:discussion}

The proof methods presented rely substantially on the assumption that the graph is regular, bipartite and has good expansion properties.
Note that all these assumptions are crucial for our application of the cluster expansion method.
This holds in particular for the assumption that the base graphs are bipartite.
For example, even for the Cartesian product of triangles the asymptotic number of independent sets is not known.
The following conjecture was made by Arras and Joos \cite{ArrasJoos23}.
\begin{conjecture}
  Let $G= \car_{i=1}^t K_3$. Then
  \[
    i(G)=(1+o(1))3 \cdot 2^{t-1} \cdot 2^{3^{t-1}} \exp((3/2)^{3^t-1}).
  \]
\end{conjecture}

\Cref{thm:general} holds for any $C_0 \geq \lambda \geq C_0 \frac{\log^2 d}{d^{1/2}}$ for sufficiently large $C_0$, the bottleneck being the application of the container lemma. Here we use a recently improved version (see \Cref{l:container}) by Jenssen, Malekshahian and Park  \cite{JeMaPa24} which works for all $\lambda$ as above. Further improvement on the range of $\lambda$ in such a container lemma would directly translate into a wider range of $\lambda$ for the results presented here (up to $\lambda=\omega\left(\frac{\log d}{d}\right)$).

The value of $\lambda$ is crucial for the structure of an independent set drawn according to the hard-core model. Indeed, for $\lambda=o(1/d)$ independent sets are quite unstructured \cite{We06}, whereas the results in this paper show that above $\lambda= \omega\left(\frac{\log^2 d}{d^{1/2}}\right)$ independent sets are highly structured, as they lie mostly on one side of the bipartition. It would be interesting to investigate whether, and under what conditions, there is a phase transition in the structure of independent sets as $\lambda$ ranges between these two regimes. Indeed, it is conjectured that the results of this paper would hold for any $\lambda= \tilde{\Omega}(1/d)$ \cite{Ga2011, JeMaPa24}, but as mentioned above only a proof of a container lemma for such $\lambda$ is missing.

In \Cref{thm:general}, large sets are required to expand by at least $(1 + \Omega(1/d))$ and this is crucial. Again, the bottleneck stems from the application of the container lemma, where this expansion is required to construct a sufficiently small family of containers already in the original approach by Galvin \cite{Ga2011}. Note that, in Cartesian product graphs such as the hypercube or even tori, we can guarantee even $(1 + \Omega(1/\sqrt{d}))$ expansion by \Cref{l:isoperimetry}.

Regarding the algorithmic aspects of \Cref{thm:general}, we wonder if there is a $\exp(O(j))$ time algorithm to compute $L_j$. In our approach, this would require finding a family of compressed $2$-linked sets of size $\exp(O(j))$ such that every $2$-linked set of size $j$ is represented by one of them. It would also require generating the corresponding clusters for each compressed set in $\exp(O(j))$ time. Both problems seem challenging.

Finally, in \cite{KrSp22}, Kronenberg and Spinka studied the asymptotic number of independent sets in the \emph{percolated hypercube} $(Q^t)_p$, by relating the expectation of the partition function of the hard-core model on a percolated graph to the partition function of the Ising model with a judicious choice of parameters.  We remark that their results are also being extended to bipartite, regular graphs with sufficient expansion in \cite{GeKaSaWdo}.

\subsection*{Acknowledgements}
The authors would like to thank Matthew Jenssen for his introduction to the cluster expansion method and fruitful discussions.
This research was funded in part by the Austrian Science Fund (FWF) [10.55776/F1002, 10.55776/P36131].
For open access purposes, the author has applied a CC BY public copyright license to any author accepted manuscript version arising from this submission.

\printbibliography

\appendix

\section{The container lemma} \label{app:container}

We dedicate this section to the derivation of \Cref{l:container} from the more general version recently proved by Jenssen, Malekshahian and Park \cite{JeMaPa24} (see Lemma 1.2 therein).

Let $\delta \geq 1$, $d_{\mathcal{E}} \geq d_{\mathcal{O}}$ and $G$ be a bipartite graph with parts $\mathcal{O}$ and $\mathcal{E}$. 
We say $G$ is \emph{$\delta$-approximately $(d_{\mathcal{E}}, d_{\mathcal{O}})$-biregular} if for any $v \in \mathcal{E}$ we have $d(v) \in [d_{\mathcal{E}}, \delta d_{\mathcal{E}}]$ and for any $v \in \mathcal{O}$ we have $d(v) \in [d_{\mathcal{O}}/\delta, d_{\mathcal{O}}]$. Fix $\mathcal{D}\in \{\mathcal{E}, \mathcal{O}\}$ to be the defect side and denote the non-defect side by $\mathcal{D}'$.
Furthermore, for $1 \leq \varphi \leq d_{\mathcal{D}'}/\delta -1$ let
\[
  m_{\varphi} \coloneqq \min\{|N(K)| \mid y \in \mathcal{D}', K \subseteq N(y), |K|> \varphi\}.
\]

As before set
\[\mathcal{G}^{\mathcal{D}}(a, b) \coloneqq \{A\subseteq \mathcal{D}: A \text{ 2-linked, } |[A]|=a, |N(A)|=b\}.\]

\begin{lemma}[Lemma 1.2, \cite{JeMaPa24}]\label{l:containernew}
  Given constants $\delta \geq 1, \delta', \delta'' >0$, there exist constants $c_0 > 0$ and $C_0 > 0$ such that the following holds. Let $G$ be a bipartite graph with parts $\mathcal{O}$ and $\mathcal{E}$ which is $\delta$-approximately $(d_{\mathcal{E}}, d_{\mathcal{O}})$-biregular with $m_{\varphi} \geq \delta'' \varphi d_{\mathcal{D}}$, where $\varphi=d_{\mathcal{D}'}/(2\delta)$. If 
  \[
  b-a \geq \max\left\{\frac{\delta' b}{d_{\mathcal{D}'}}, \frac{c_0 d_{\mathcal{D}}}{\log^2 d_{\mathcal{D}}}\right\}, \]
then for every $\lambda \geq (C_0 \log^2 d_{\mathcal{D}})/\sqrt{d_{\mathcal{D}}}$ it holds that
  \[
    \sum_{A \in \mathcal{G}^{\mathcal{D}}(a, b)} \lambda^{|A|} \leq |\mathcal{D}'| (1+\lambda)^b \exp\left(-\frac{(b-a) \log^2 d}{6d_{\mathcal{D}}}\right).
  \]
\end{lemma}

We derive \Cref{l:container} by checking the assumptions of \Cref{l:containernew}.
\begin{proof}[Proof of \Cref{l:container}]
  Recall that $c > 0$ is such that every $X \subseteq \mathcal{D}$ with $|X|\leq n/4$ satisfies $|N(X)| \geq \left(1+\frac{c}{d}\right)|X|$.
  We take $\delta = 1$, $\delta' = c/2$ and $\delta'' = 1/(2\Delta_2)$, and apply \Cref{l:containernew} to obtain constants $c_0 > 0$ and $C_0 > 0$.
  Let $G$ be a graph fulfilling the assumptions of \Cref{l:container}, and let $\mathcal{O},\mathcal{E}$ be the partition classes of $G$.
  Since $G$ is $d$-regular, we have $d_{\mathcal{E}}=d_{\mathcal{O}}=d$ and therefore $G$ is $\delta$-approximately $(d,d)$-biregular. It remains to verify that $G$ satisfies the other conditions in the statement.
  
  Note that $\varphi = d/2$. We start by showing that $m_\varphi \geq \Delta_2 \cdot d^2/(4\Delta_2)$.
  Let $y \in \mathcal{D}'$ and consider $K \subseteq N(y)$ of size $d/2 < |K| \leq d$. We claim that $|N(K)| \geq d^2/(2 \Delta_2)$.
  Indeed, for each vertex $x \in N(K) \setminus \{ y \}$ note that $|N(x) \cap K| \leq \Delta_2$ since otherwise the co-degree of $x$ and $y$ would exceed $\Delta_2$. Since the graph is $d$-regular, we have 
  \[ |E(K, N(K) \setminus \{y\})| = |E(K, N(K))| - |K| = (d-1)|K| \] and thus $|N(K)|\geq 1 + (d-1)|K|/\Delta_2$. As $|K|> d/2$ and $d$ is large, this proves the claim.

  We now show the required bounds on $b-a$. Suppose $\mathcal{G}^{\mathcal{D}}(a, b)$ is non-empty for $d^2 \leq a \leq n/4$. Then, taking $A \in \mathcal{G}^{\mathcal{D}}(a, b)$, we have $|[A]| = a \leq n/4$ and so
  \[ 
  b = |N(A)| = |N([A])| \geq \left(1+\frac{c}{d}\right)a. 
  \]
  We may assume $c/d \leq 1$, which implies
  \[
    a \leq \frac{b}{1+c/d} \leq \left(1-\frac{c}{2d}\right)b = \left(1 - \frac{\delta'}{d}\right)b,
  \]
  showing the first bound.
  As $a \geq d^2$, we have
  \[
    b - a \geq \frac{ca}{d} \geq \frac{c}{d} \cdot d^2 \geq \frac{c_0d}{\log^2 d},
  \]
  since we may assume $d$ is large enough. This proves the second bound.
  Thus, by \Cref{l:containernew}, for $\lambda \geq \frac{C_0 \log^2 d}{d^{1/2}}$,
  \[
    \sum_{A\in \mathcal{G}^{\mathcal{D}}(a, b)}\frac{\lambda^{|A|}}{(1+\lambda)^b}\leq \frac{n}{2}\exp\left(-\frac{(b-a) \log^2 d}{6d}\right),
  \]
  proving \Cref{l:container}.
\end{proof}

\section{Deferred calculations} \label{sec:calculations}

First recall that $\lambda_1 \coloneqq \min(\lambda, 1)$ and
\begin{align}
  \gamma(\ell) &\coloneqq     \left\{\begin{array} {c@{\quad \textup{if} \quad}l}
    (1+ \lambda)^{d - \Delta_2 \ell} / d^8 & \ell \leq d/ \log \log d, \\[+1ex]
    (1+ \lambda)^{\frac{16 \log d}{\lambda_1}} & d/ \log \log d < \ell \le d^3 \log n, \\[+1ex]
    \exp(1/d^{2}) & \ell > d^3 \log n.
  \end{array}\right.
\end{align}

Let us prove \Cref{c:gamma_monotone}.
\begin{claim}
    The function $\ell \mapsto \gamma(\ell)$ is non-increasing.
\end{claim}
\begin{proof}
The claim is clear for each regime of $\ell$ since in the first regime $\gamma$ is decreasing in $\ell$ and in the other two regimes $\gamma$ is constant.
At the first transition when $\ell=d/ \log \log d$, we need to check that
\[
\gamma(d/ \log \log d)= \frac{(1+\lambda)^{d-\Delta_2 d/ \log \log d}}{d^{8}} \geq (1+\lambda)^{\sqrt{d}/2} \geq (1+\lambda)^{\frac{16 \log d}{\lambda_1}},
\]
which is true as $d \to \infty$ since $\lambda \geq C_0 \frac{\log^2 d}{d^{1/2}}$.
At the second transition when $\ell=d^3 \log n$, we need to check that
\[
(1+\lambda)^{\frac{16 \log d}{\lambda_1}} \geq \exp(1/d^{2}).
\]
Since $\lambda$ is bounded from above, this is clear if $\lambda \geq 1$. Otherwise, we can rewrite the left-hand side as
\[
(1+\lambda)^{\frac{\log d}{\lambda_1}} = \exp\left(\frac{16\log d}{\lambda} \log (1+\lambda)\right) \geq 2^{16 \log d},
\]
which is indeed much larger than $\exp(1/d^{2})$.
\end{proof}

\begin{claim}\label{c:gamma-constant-k}
    For constant $k \in \mathbb{N}$, it holds that $k \log \gamma(k) \leq \ell \log \gamma(\ell)$ for every $k < \ell$.
\end{claim}
\begin{proof}
Let $k \in \mathbb{N}$ be a fixed constant.
Consider first $k < \ell \leq d/ \log \log d$. In this regime we have
\begin{align*}
    \ell \log (\gamma(\ell)) &= \ell (d-\Delta_2\ell) \log(1+\lambda) - 8 \ell \log d \\
    & = (1+o(1))\ell d  \log(1+\lambda)\\
    & \geq kd \log(1+\lambda) \geq k \log (\gamma(k)),
\end{align*}
where the second equality follows using $\lambda \geq C_0 \frac{\log^2 d}{d^{1/2}}$ implying that the first term is dominating.

For $d/ \log \log d < \ell \leq d^3 \log n$ using that $\lambda$ is bounded from above we have
\begin{align*}
    \ell \log (\gamma(\ell)) &\geq \frac{16 \log d}{\lambda_1 \log \log d} d \log(1+\lambda) \geq k d \log (1+\lambda) \geq k \log (\gamma(k)).
\end{align*}

For $\ell > d^3 \log n$ we have
\[
    \ell \log (\gamma(\ell)) \geq d^3 \log n \cdot \frac{1}{d^{2}} = d \log n \geq k d \log (1+\lambda) \geq k \log (\gamma(k)),
\]
finishing the proof.
\end{proof}

\section{Cluster expansion coefficients for $\car_{i=1}^t K_{s,s}$}\label{app:coefficients}
Letting $\hat{L}_j = L_j \cdot 2^{jst} / (2s)^t$, the computed coefficients are listed below.
\begin{align*}
  \hat{L}_1 &= \frac{1}{2}, \\
  \hat{L}_2 &= \frac{3}{8} \, s^{2} t^{2} + \frac{1}{8} \, {\left(2 {\left(s - 1\right)} 2^{s} - 3 \, s^{2} - 2 \, s + 2\right)} t - \frac{1}{4}, \\
  \hat{L}_3 &= \frac{9}{16} \, s^{4} t^{4} - \frac{1}{24} \, {\left(- 18 \, {\left(s^{3} - s^{2}\right)} 2^{s} + 27 \, s^{4} + 28 \, s^{3} - 18 \, s^{2}\right)} t^{3} \\
            &\quad + \frac{1}{16} \, {\left(4 \, {\left(s^{2} - 2 \, s + 1\right)} 2^{2 \, s} - 8 \, {\left(s^{3} - 2 \, s + 1\right)} 2^{s} + 9 \, s^{4} - 4 \, s^{3} - 2 \, s^{2} - 8 \, s + 4\right)} t^{2} \\
            &\quad + \frac{1}{24} \, {\left(- 2 \, {\left(2 \, s^{2} - 3 \, s + 1\right)} 2^{2 \, s} - 6 \, {\left(s^{3} - 2 \, s^{2} + 3 \, s - 2\right)} 2^{s} + 34 \, s^{3} - 11 \, s^{2} + 12 \, s - 10\right)} t \\
            &\quad + \frac{1}{6},\\
  \hat{L}_4 &= \frac{9}{8} \, s^{6} t^{6} - \frac{3}{8} \, {\left(- 6 \, {\left(s^{5} - s^{4}\right)} 2^{s} + 9 \, s^{6} + 11 \, s^{5} - 6 \, s^{4} \right)} t^{5} \\
            &\quad + \frac{1}{64} \, \Bigl(96 \, {\left(s^{4} - 2 \, s^{3} + s^{2}\right)} 2^{2 \, s} - 8 \, {\left(27 \, s^{5} + 7 \, s^{4} - 58 \, s^{3} + 24 \, s^{2}\right)} 2^{s} \\
            &\qquad\qquad + 216 \, s^{6} + 120 \, s^{5} + 349 \, s^{4} - 272 \, s^{3} + 96 \, s^{2}\Bigr) t^{4} \\
            &\quad - \frac{1}{96} \, \Bigl(- 32 \, {\left(s^{3} - 3 \, s^{2} + 3 \, s - 1\right)} 2^{3 \, s} + 6 \, {\left(24 \, s^{4} - 35 \, s^{3} - 21 \, s^{2} + 48 \, s - 16\right)} 2^{2 \, s} \\
            &\qquad\qquad - 12 \, {\left(s^{5} + 47 \, s^{4} - 64 \, s^{3} + 24 \, s - 8\right)} 2^{s} \\
            &\qquad\qquad + 108 \, s^{6} - 1314 \, s^{5} + 4173 \, s^{4} - 860 \, s^{3} + 30 \, s^{2} + 96 \, s - 32\Bigr) t^{3} \\
            &\quad + \frac{1}{64} \, \Bigl(4 \, {\left(s^{2} - 2 \, s + 1\right)} 2^{4 \, s} - 32 \, {\left(s^{3} - 2 \, s^{2} + s\right)} 2^{3 \, s} \\
            &\qquad\qquad + 4 \, {\left(30 \, s^{4} - 69 \, s^{3} + 5 \, s^{2} + 48 \, s - 14\right)} 2^{2 \, s} \\
            &\qquad\qquad + 8 \, {\left(7 \, s^{5} - 225 \, s^{4} + 301 \, s^{3} - 63 \, s^{2} - 32 \, s + 12\right)} 2^{s} \\
            &\qquad\qquad + 162 \, s^{6} - 2028 \, s^{5} + 8723 \, s^{4} - 3920 \, s^{3} + 686 \, s^{2} + 104 \, s - 44\Bigr) t^{2} \\
            &\quad - \frac{1}{96} \, \Bigl(6 \, {\left(s^{2} - 2 \, s + 1\right)} 2^{4 \, s} - 2 \, {\left(9 \, s^{3} - 6 \, s^{2} - 13 \, s + 10\right)} 2^{3 \, s} \\
            &\qquad\qquad + 2 \, {\left(90 \, s^{4} - 239 \, s^{3} + 132 \, s^{2} + 11 \, s + 6\right)} 2^{2 \, s} \\
            &\qquad\qquad - 12 \, {\left(s^{5} + 203 \, s^{4} - 293 \, s^{3} + 85 \, s^{2} + 8 \, s - 4\right)} 2^{s} \\
            &\qquad\qquad + 243 \, s^{6} - 1944 \, s^{5} + 9651 \, s^{4} - 5416 \, s^{3} + 1143 \, s^{2} + 60 \, s - 46\Bigr) t \\
            &\quad - \frac{1}{8}.
\end{align*}

The above coefficients, as well as $L_5$ and $L_6$ (which were omitted for space reasons and can be found in the GitHub repository), were computed for general $\lambda$ and then specialised for $\lambda = 1$.

\end{document}